\documentclass[twoside,journal,final,twocolumn]{IEEEtran}

\usepackage{cite}
\usepackage{color,array}
\usepackage{amsmath,amssymb,amsfonts}
\usepackage{graphicx}
\usepackage{algorithm,algorithmic}
\usepackage{hyperref}
\usepackage{textcomp}
\usepackage{optidef}
\usepackage{balance}
\usepackage{fdsymbol}

\usepackage[caption=false,font=normalsize,labelfont=sf,textfont=sf]{subfig}
\usepackage{stfloats}
\usepackage{url}
\usepackage{verbatim}

\usepackage{multirow}
\usepackage{cuted}
\usepackage{flushend}

\newcommand{\lp}{ \left( }
\newcommand{\rp}{ \right) }

\newcommand{\xk}{ x_k }
\newcommand{\yk}{ y_k }
\newcommand{\zk}{ z_k }
\newcommand{\uk}{ u_k }

\newcommand{\dxk}{ \delta x_k }

\newcommand{\xkp}{ x_{k+1} }
\newcommand{\zkp}{ z_{k+1} }
\newcommand{\dxkp}{ \delta x_{k+1} }

\newcommand{\Rn}{ {\mathbb{R}}^n }
\newcommand{\Rm}{ {\mathbb{R}}^m }
\newcommand{\Rp}{ {\mathbb{R}}^p }
\newcommand{\Rnn}{ {\mathbb{R}}^{n \times n}}

\newcommand{\Rmm}{ {\mathbb{R}}^{m \times m}}
\newcommand{\Rmn}{ {\mathbb{R}}^{m \times n}}
\newcommand{\hx}{\hat{x}}
\newcommand{\hu}{\hat{u}}
\newcommand{\hr}{\hat{r}}
\newcommand{\tx}{\tilde{x}}
\newcommand{\tu}{\tilde{u}}
\newcommand{\R}{R}
\newcommand{\Q}{Q}

\newcommand{\ueq}{u_{\text{eq}}}
\newcommand{\xeq}{x_{\text{eq}}}

\newtheorem{theorem}{Theorem}
\newtheorem{lemma}{Lemma}
\begin{document}

\title{A Data-Driven Algorithm for Model-Free Control Synthesis}

\author{Sean Bowerfind,
        Matthew R. Kirchner,~\IEEEmembership{Member,~IEEE,}
        and Gary Hewer

        
\thanks{This research was supported by the Office of Naval Research under Grant: N00014-24-1-2322. The views expressed in this article are those of the author and do not necessarily reflect the official policy or position of the Air Force, the Department of Defense or the U.S. Government. \textit{(Corresponding author: M. R. Kirchner.)}}


\thanks{S. Bowerfind and M. R. Kirchner are with the Department of Electrical and Computer Engineering, Auburn University, Auburn, AL 36832 USA (e-mail: \href{mailto:srb0063@auburn.edu}{srb0063@auburn.edu}; \href{mailto:kirchner@auburn.edu}{kirchner@auburn.edu}); G. Hewer was formally with (now retired) the Physics and Computational Sciences Division, Naval Air Warfare Center Weapons Division, China Lake, CA 93555 USA (email: gary.a.hewer.civ@us.navy.mil\href{gary.a.hewer.civ@us.navy.mil}).}}


\maketitle

\begin{abstract}
Presented is an algorithm to synthesize the optimal infinite-horizon LQR feedback controller for continuous-time systems. The algorithm does not require knowledge of the system dynamics but instead uses only a finite-length sampling of input-output data. A necessary condition that relates the optimal LQR gain to any arbitrary solution trajectory of the system is presented. An algorithm using this necessary condition, based on constrained optimization, is developed that estimates the LQR gain matrix. In addition to calculating the standard feedback gain matrix, a feedforward gain can be found to implement a reference tracking controller. This paper presents a theoretical justification for the method and shows several examples, including a validation test flight on a real scale aircraft with unknown dynamics.
\end{abstract}


\section{Introduction} \label{sec:introduction}

\IEEEPARstart{T}{he} linear quadratic regulator (LQR) has been established as one of the most useful state-space control design methodologies, as it provides many beneficial properties, including simple hardware implementation and robustness guarantees with respect to both phase and gain margin \cite{safonov1977gain,lehtomaki1981robustness}. Classical LQR design assumes linear dynamics and relies on complete knowledge of system matrices, $A$ and $B$, to compute the optimal feedback gain matrix. These matrices typically come from either direct knowledge of an inherently linear system or through a linearization from known nonlinear dynamics. 

For many systems, direct knowledge of the system is unknown a-priori, but collection of data by operating such systems is readily available. This motivates our so-called model-free approach. By this we assume the system is governed by a linear ODE of the form
\[
\dot{x}(t) =  Ax(t) + Bu(t),
\]
where the system matrices $A$ and $B$ are unknown to the designer. The goal is to design the feedback gain matrix using only observations of the state vector, $x\lp t \rp$, after an input sequence, $u\lp t \rp$, has been applied to the system. The gain matrix should be the same as the one calculated if the system matrices are known.

When the system is known to be linear, model identification techniques can be attempted to first find the $A$ and $B$ matrices before solving for the gain matrix by using the algebraic Riccati equation. By contrast, our model-free approach can be seen as streamlining the design procedure by going directly from data to LQR controller, skipping the model identification task entirely.

A more common situation arises when the system is inherently nonlinear and an LQR controller is calculated by using $A$ and $B$ matrices found through linearization at an equilibrium point \cite[Chap. 21, p. 201]{hespanha2018linear}. In this case, model identification of the original nonlinear model is more difficult. However, in a model-free design paradigm, only a collection of data where the system is operating near the equilibrium point is required. Therefore, a feedback gain matrix found from a model-free approach should be similar to the controller generated by a linearization when the model is known. 

One of the most important applications of the proposed controller design method lies within aerospace testing of small unmanned aerial vehicles (UAVs). The fidelity of an aircraft model depends heavily on the ability to accurately calculate the external forces and moments acting on the aircraft for various flight conditions, i.e. $F=ma$ and generalizations thereof \cite[Chap. 2]{stevens_aircraft_2015}. Common methods to determine these forces are computational fluid dynamics studies \cite{andersonComputationalFluidDynamics2006}, wind tunnel tests \cite{andersonFundamentalsAerodynamics2017}, and static propulsion tests \cite{zipfelModelingSimulationAerospace2007}.  For UAVs, particularly the low-cost and subscale aircraft considered in this study, the application of the above high-fidelity modeling techniques becomes time-consuming and economically infeasible. These challenges in UAV modeling are a key driving force of the present research.

There have been recent attempts to compute LQR directly from data \cite{vrabie_optimal_2013}, typically formulating with discrete-time systems. Most previous works are based on the reinforcement learning paradigm \cite{kaelbling1996reinforcement,sutton1998reinforcement,lewis_reinforcement_2011} and these family of methods rely on a process called \emph{policy iteration} \footnote{For a discussion between on-policy and off-policy methods, see \cite[Sec. 5.4, p. 100]{sutton1998reinforcement}.}. This methodology iteratively updates the policy (here the feedback gain matrix) and distinctive within this class of methods is that policy updates and data collection are coupled: Data used for the next iterate must come from the previous policy as input. Typically, this coupling is implemented with a required computer simulation of the desired platform \footnote{One such example uses the phrase: ``interactions with the system'' \cite{hall2011reinforcement}.}. When reinforcement learning approaches are applied to a model-free control design problem \cite{vamvoudakis_q-learning_2017,luo2014q,lee2012integral,lee2014integral}, a contradiction is created: If the user has enough knowledge of the system to construct a simulation, then the user has enough knowledge for traditional design techniques, and model-free approaches are unnecessary.

Another online method is presented in \cite{vamvoudakis_q-learning_2017}. Their adaptive control technique, titled actor-critic, updates candidate LQR control solutions but requires careful tuning of hyperparameters and can be susceptible to numerical instabilities. 

Offline approaches have also been investigated. Most notable is the work of \cite{farjadnasab2022model}. This work assumes discrete-time dynamics and uses this representation to remove the dependence on $A$. This formulation can induce a noise sensitivity in the data, which becomes more prominent when the sample time in the data becomes small. 

This article presents a model-free approach to synthesize an LQR controller directly from data. It is presumed that the data collected are the control sequence that drives the system and the resulting noisy observations of the state\footnote{Or observation of a function of the state, i.e. $y=Cx$.}. Since the optimal cost-to-go for the LQR problem has a known parametric form, the cost-to-go can be evaluated along any arbitrary trajectory. This cost-to-go must satisfy a specific dynamic constraint and is shown in this paper to be a necessary condition for the optimal gain matrix, $K$. This necessary condition is utilized to construct a nonlinear programming (NLP) problem, whereby the above dynamic constraint is enforced along the observed data and the solution of which is the optimal closed-loop LQR feedback gain, $K$. This NLP is formulated implicitly with respect to the state trajectory to ensure robustness to noisy data.  

The method is further extended from a model-free regulator problem to the model-free reference tracking problem. This necessitates the computation of a feedforward gain matrix, $F$, in addition to $K$. This paper develops a generalization to the dynamic constraint, which includes an arbitrary constant as an additional input, resulting in a new NLP problem. Solving this NLP simultaneously achieves the optimal feedback and feedforward gains $K$ and $F$, without knowledge of the system dynamics. 

Lastly, a problem wherein a subset of the system dynamics is available, referred to as a ``mixed-model'', is considered. As an example, a system with unknown dynamics is driven by an actuator, where the dynamics of the actuator system is readily available. This mixed-model problem is formulated by integrating known information about the actuator system with the remaining unknown dynamics, and is used to compute the optimal closed-loop LQR gain.

We demonstrate the proposed method on two examples. The first uses data generated from a benchmark linear system, inspired by a McDonnell-Douglas A-4D Skyhawk. This example serves to validate the controller synthesis procedure by comparing the resultant model-free controller gain with the ground truth controller, since the $A$ and $B$ matrices are known. This method also validates two model-free controller variations, which are model-free reference tracking and mixed-model reference tracking.

Second, the versatility of the proposed model-free controller is demonstrated through real-world flight tests. A model-free controller is implemented to control the roll angle of a subscale UAV. The entire process, from initial data collection flights to controller validation, is detailed to demonstrate the successful application of the model-free controller onboard a real aircraft. The data used to design the controller was obtained directly from sensor measurements under windy flight conditions, which confirms the method's robustness to realistic noise environments. The model-free controller was flown during challenging flight conditions and the aircraft successfully achieved all test objectives.

We summarize the contributions below:

\begin{enumerate}
\item A method to synthesize an LQR gain matrix directly from data is presented.
This method does so without knowledge of the system matrices $A$
and\textbf{ $B$}. It avoids policy iteration and uses data that is
decoupled from the policy estimation.
\item A necessary condition is derived and proved that relates state trajectories to the optimal LQR gain and does so under a general set of regularity assumptions that includes discontinuous inputs.
\item The algorithm is generalized to find the feedforward gain matrix for
LQR reference tracking problems.
\item Sensor noise is directly modeled to improve robustness to the noise encountered when collecting data from real systems.
\item The approach was validated on a real scale aircraft under challenging
environmental conditions.
\end{enumerate}

\section{LQR Optimal Control Fundamentals}\label{sec:LQR fundamentals}

Consider the following linear time-invariant (LTI) continuous time dynamical system
\begin{equation}
    \begin{cases}
        \dot{x}(t) = f\lp x,u\rp :=  Ax(t) + Bu(t), \\
        x\lp 0\rp = x_0,
    \end{cases}
    \label{eq:LTIsys}
\end{equation}
for $t\in \left[0,T\right]$ where $x(t) \in \Rn$ is the system state and $u(t) \in \Rm$ is the control input and is assumed drawn from the set:
\[
u\left(\cdot\right)\in\mathcal{U}\left(T\right)=\left\{ \varphi:\left[0,T\right]\mapsto\mathbb{R}^{m}|\varphi\left(\cdot\right)\,\text{is measurable}\right\} .
\]
One can assume without loss of generality that the initial time is always $t_0=0$ since the system \eqref{eq:LTIsys} is time invariant. As a shorthand, we denote the state trajectory 
$\left[0,T \right]\ni t\mapsto x\lp t;x_0,u\lp\cdot\rp\rp\in\Rn$ by $x\lp t\rp$. When $t$ appears, we assume the above definition that $x\lp t \rp$ evolves in time with control sequence $u\lp\cdot\rp$,
according to \eqref{eq:LTIsys} starting from initial state $x_{0}$ at $t=0$. When $x$ appears without dependence on $t$, we are referring to a specific state without regard to an entire trajectory. 

The infinite time horizon ($T\rightarrow\infty$) optimal control problem seeks to find a feedback controller of the form
\[
u\left(t\right)=-Kx\left(t\right),
\]
such that it optimizes the following cost functional:
\begin{align}
J\lp x_{0},u\lp\cdot\rp\rp &=\int_{0}^{\infty} x\lp t\rp^{\top}\Q x\lp t\rp+u\lp t\rp^{\top}\R u\lp t\rp dt,\label{eq:infinite cost funct}
\end{align}
with $x(0)=x_0$ and where $\Rnn\ni \Q=\Q^\top \succeq 0$ and $\Rmm \ni \R=\R^\top \succ 0$ are the user-supplied weight matrices.

The value function $V:\Rn \rightarrow {\mathbb{R}}$ is defined as the minimum cost, $J$, among all admissible controls for an initial state $x_0$ given as
\begin{equation}
     V\lp x_0 \rp = \inf_{u\lp\cdot\rp\in\mathcal{U}}  J\lp x_0, u\lp(\cdot\rp \rp,
     \label{eq:value funct}
\end{equation}
subject to the dynamic system constraint \eqref{eq:LTIsys}. Finding a feedback control policy that minimizes \eqref{eq:value funct} is known as the infinite horizon linear quadratic regulator (LQR) problem. It is well known that the value function takes a quadratic form
\begin{equation}
    V(x) = x^\top P x,
    \label{eq:optimal value funct}
\end{equation}
where $\Rnn \ni P=P^\top \succeq 0$. Without loss of generality, $Q$ can be written as
\[
\Q=C^{\top}C,
\]
since $\Q$ is symmetric positive semi-definite \cite[Thm. 7.2.7, p. 440]{horn2012matrix}. This representation is most common when constructing a cost functional \eqref{eq:infinite cost funct} in terms of a controlled output, $z=Cx$ \cite[Chap. 21, p. 201]{hespanha2018linear}.

It is well known \cite{kailath1980linear} that when the pair $\lp A,B\rp$ is stabilizable and the pair $\lp A,C\rp$ is detectable, then $P$ is the unique positive semi-definite solution of the algebraic Riccati (ARE) equation
\begin{equation}
    A^\top P + PA - PB\R^{-1}B^\top P +\Q = 0,
    \label{eq:ARE}
\end{equation}
and the optimal closed loop control is
\begin{equation}
u^{*}=-Kx,\label{eq:optimal control}
\end{equation}
where
\begin{equation}
    K=\R^{-1}B^{\top}P.\label{eq:LQR optimal K}
\end{equation}

\section{Model-Free Necessary Conditions for Optimality}\label{sec:model-free optimality}
The following theorem establishes a necessary condition that $P$ in $\lp\ref{eq:optimal value funct}\rp$ and $K$ in $\lp\ref{eq:LQR optimal K}\rp$ must satisfy with respect to any arbitrary solution of the system ODE defined in $\lp\ref{eq:LTIsys}\rp$ for an LQR controller. This will be utilized in developing a numeric algorithm to synthesize the optimal LQR controller from a finite-length segment of empirically collected system data without having to know the matrices $A,B$ that define the dynamical model of the system.

\begin{theorem}\label{thm: main thm}
Suppose the trajectory pair $\left[0,T\right]\ni t \mapsto\left\{ \hx(t), \hu(t)\right\}\in\mathbb{R}^n\times\mathbb{R}^m$, is a solution to $\lp\ref{eq:LTIsys}\rp$ in the Carath\'{e}odory sense, $P=P^\top \succeq 0$ is the solution to the Algebraic Riccati Equation $\lp\ref{eq:ARE}\rp$, and $K$ is the LQR optimal gain matrix given in $\lp\ref{eq:LQR optimal K}\rp$. Then the scalar valued function $\left[0,T\right]\ni t\mapsto V\lp \hx\lp t\rp\rp\in{\mathbb{R}}$, where $V$ is defined by \eqref{eq:optimal value funct}, satisfies       
\begin{equation}
    \begin{split}
            \dot{V}\lp \hx\lp t\rp\rp &  =\hx\lp t\rp^{\top}K^{\top}\R K\hx\lp t\rp\\ 
            &-\hx\lp t\rp^{\top}\Q \hx\lp t\rp+2\hx\lp t\rp^{\top}K^{\top}\R \hu\lp t\rp,
    \end{split}
    \label{eq:Vdot optimality condition}
    \end{equation} 
almost everywhere on $t\in [0,T]$.
\end{theorem}

\begin{IEEEproof}
Since the dynamics, $f\lp x,u\rp=Ax+Bu$, are Lipschitz continuous in $x$ and the control sequence $u\lp\cdot\rp$ is measurable in $t$ for fixed $x$, then $\lp\ref{eq:LTIsys}\rp$ satisfies the Carath\'{e}odory conditions \cite{filippov1988differential} and there exists an absolutely continuous solution $x\lp t;x_0,u\lp\cdot\rp \rp$ that satisfies $\lp\ref{eq:LTIsys}\rp$ almost everywhere $t\in \left[ 0,T \right]$ \cite[Chap. 2, Theorem 1.1, p. 43]{coddington1956theory}. 

Since by assumption, $\hx\lp t \rp$ is a Carath\'{e}odory solution, it is absolutely continuous and implies that $\dot{\hx}\lp t \rp$ exists almost everywhere on $t\in\left[0,T\right]$ \cite[Chap. 5, Corollary 12, p. 109]{royden1988real}. This implies that the composition $\left[ 0,T\right]\ni t \mapsto V\lp \hx\lp t \rp \rp$ is differentiable almost everywhere \cite[Theorem 2.12]{kolavr2025chain} and found from the standard chain rule \cite[Theorem 2]{serrin1969general} as
\begin{align}
\dot{V}\lp \hx\lp t\rp\rp & =\frac{\partial V}{\partial x}\dot{\hx}\lp t\rp \nonumber\\
 & =\frac{\partial V}{\partial x}\lp A\hx\lp t\rp+B\hu\lp t\rp\rp\nonumber \\
 & =2\hx\lp t\rp^{\top}PA\hx\lp t\rp+2\hx\lp t\rp^{\top}PB\hu\lp t\rp,\label{eq:implicit value function}
\end{align}
almost everywhere $t\in \left[ 0,T \right]$, where the last line is from
\[
\frac{\partial}{\partial x}\left\{ V\lp x \rp\right\} =\frac{\partial}{\partial x}\left\{ x^{\top}Px\right\} =2x^{\top}P.
\]

Under a set of mild Lipschitz continuity assumptions, there exists a unique value function \eqref{eq:value funct} that satisfies, for all $x\in\Rn$, the following stationary
(time-independent) Hamilton--Jacobi partial differential equation (PDE)
\begin{equation}
\underset{u\in\Rm}{\min}H\lp x,u,\frac{\partial V}{\partial x}\rp=0,\label{eq:HJB equation for inf LQR}
\end{equation}
where
\begin{align*}
H\lp x,u,p\rp:=\left\langle p, f\lp x,u\rp \right\rangle + x^{\top}\Q x+u^{\top}\R u.
\end{align*}
This implies that for the system given in \eqref{eq:LTIsys}, the cost functional given in \eqref{eq:infinite cost funct}, and noting that $P$ satisfies $\lp \ref{eq:ARE}\rp$ we have
\begin{align}
H\lp x,u,\frac{\partial V}{\partial x}\rp = & \left\langle\frac{\partial V}{\partial x}, Ax+Bu\right\rangle +x^{\top}\Q x+u^{\top}\R u\nonumber \\
 =&2x^{\top}PAx+2x{}^{\top}PBu + x^{\top}\Q x+u^{\top}\R u.\label{eq:raw hamiltonian}
\end{align}
Since $P$ must satisfy $\lp \ref{eq:ARE}\rp$, the  optimal control is given by
\begin{equation}
u^{*}=-\R^{-1}B^{\top}Px.\label{eq:optimal control hamiltonian}
\end{equation}
This implies that, after substituting \eqref{eq:optimal control hamiltonian} into \eqref{eq:raw hamiltonian}, the equality in \eqref{eq:HJB equation for inf LQR} can be written as
\[
2x^{\top}PAx-x^{\top}PB\R^{-1}B^{\top}Px+x^{\top}\Q x=0,
\]
which gives the expression
\begin{equation}
2x^{\top}PAx=x^{\top}PB\R^{-1}B^{\top}Px-x^{\top}\Q x.\label{eq:Sub A out}
\end{equation}
When \eqref{eq:Sub A out} is evaluated along the trajectory $\hx\lp t\rp$ and substituted into \eqref{eq:implicit value function}, the dependence on the system matrix, $A$, is removed and results in
\begin{align}
\dot{V}\lp \hx\lp t\rp\rp & =\hx\lp t\rp^{\top}PB\R^{-1}B^{\top}P\hx\lp t\rp\nonumber \\
 & -\hx\lp t\rp^{\top}\Q\hx\lp t\rp+2\hx\lp t\rp^{\top}PB\hu\lp t\rp.
\end{align}
Performing the substitution with $K = \R^{-1}B^\top P$ gives
\begin{align}
\dot{V}\lp \hx\lp t\rp\rp & =\hx\lp t\rp^{\top}K^\top \R K\hx\lp t\rp\nonumber \\
 & -\hx\lp t\rp^{\top}\Q\hx\lp t\rp+2\hx\lp t\rp^{\top}K^\top \R\hu\lp t\rp,\label{eq:cont. time equality constraint}
\end{align}
which holds almost everywhere on $t\in\left[0,T \right]$, and the result is shown.
\end{IEEEproof}

\section{Data-Driven Control Synthesis}\label{sec:data driven control}

Consider a trajectory segment of finite length on $\left[0,T\right]$ for some $0<T<\infty$, whereby $N+1$ samples are observed on the time grid
\[
\pi^{N}=\left\{ t_{k}:k=0,\ldots ,N\right\} .
\]
The samples observed are the control inputs, $u\lp t_{k}\rp$, and output measurements $y\lp t_{k}\rp$, where direct observation of the full state is not assumed, but rather a classical linear observation model perturbed by noise
\[
y\lp t_{k}\rp =x\lp t_{k}\rp+\epsilon,
\]
where $\epsilon$ is zero mean i.i.d. Gaussian noise. We use the shorthand
\begin{equation}\label{eq:notation}
\begin{cases}
\xk:=x\lp t_{k}\rp,\\
\yk:=y\lp t_{k}\rp,\\
\uk:=u\lp t_{k}\rp,
\end{cases}
\end{equation}
for any $k\in\pi^{N}$. Additionally, denoting $X$ as the collection of all latent (unobserved) states defined by
\begin{align}
    X:=\left[x_{0},\ldots,x_{N}\right],
\label{eq:bigX}
\end{align}
denoting $U$ as the collection of all known control inputs by
\begin{align}
    U:=\left[u_{0},\ldots,u_{N}\right],
\label{eq:bigU}
\end{align}
and denoting $Y$ as the collection of all output measurements defined by
\begin{align}
   Y:=\left[y_{0},\ldots,y_{N}\right].
\label{eq:bigY} 
\end{align}
The rate in \eqref{eq:cont. time equality constraint} is discretized, utilizing the known structure of $V$, so that for any time $t_{k}$
\[
\dot{V}\lp x\lp t_{k}\rp\rp=2\xk^{\top}P\dot{x}_k\approx 2\xk^{\top}P D_{k}^{x},
\]
where $D_{k}^{x}$ is a finite difference approximation of $ \dot{x}\lp t_{k}\rp$. There are no requirements on any specific finite differencing form or method, so for simplicity of what is to follow, the first-order forward Euler method
\begin{equation}
D_{k}^{x}=\frac{x_{k+1}-x_{k}}{\Delta t},\label{eq:finte diff in V}
\end{equation}
is used, with $\pi^{N}$ consisting of uniformly spaced time samples\footnote{Recall, $N$ is defined as the number of samples \emph{not} including the index 0.}
\[
\Delta t=\frac{T}{N}.
\]
Many approximations exist for $D_{k}^{x}$, for example higher-order finite differencing schemes \cite{mathews2004numerical} as well as trajectory representations for both uniformly spaced \cite{cichella2018bernstein} and non-uniform time grids \cite{elnagar1995pseudospectral}. These can increase computational accuracy and can be applied to what
follows but are outside the scope of this work.

Proceeding with $D_{k}^{x}$ as in \eqref{eq:finte diff in V}, $\dot{V}\lp x\lp t_{k}\rp\rp$ is approximated by
\[
\dot{V}\lp x\lp t_{k}\rp\rp=\frac{2}{\Delta t}\lp x_{k}^{\top}Px_{k+1}-x_{k}^{\top}Px_{k}\rp.
\]
The equality \eqref{eq:cont. time equality constraint} from Theorem \ref{thm: main thm} implies the following holds for all $k=0,\ldots N-1$: 
\begin{align}
\frac{2}{\Delta t}\lp x_{k}^{\top}Px_{k+1}-x_{k}^{\top}Px_{k}\rp & =x_{k}^{\top}K^\top \R Kx_{k}\nonumber \\
 & -x_{k}^{\top}\Q x_{k}+2x{}_{k}^{\top}K^\top \R u_{k}.\label{eq:discrete equality}
\end{align}
Given the set of observations $\left\{ \uk,\yk\right\} _{k=0,\ldots,N}$, the goal is to compute the pair $\lp P,K\rp$, with $P$ being symmetric positive semi-definite, that satisfies \eqref{eq:discrete equality}. This is achieved by the following constrained optimization:
\begin{equation}
\begin{cases}
\underset{L,K,X}{\min} & \left\Vert \text{vec}\left(Y\right)-\text{vec}\left(X\right)\right\Vert _{2}^{2}\\ 
\\
\underset{\forall k=0,\ldots,N-1}{\text{subject to}} & \frac{2}{\Delta t}\lp\xk^{\top}L^{\top}L\xkp-\xk^{\top}L^{\top}L\xk\rp\\
 & -\xk^{\top}K^\top \R K\xk\\
 & +\xk^{\top}\Q\xk-2\xk^{\top}K^\top \R\uk=0,
\end{cases}\label{eq:NLP}
\end{equation}
where $L$ is defined as $P=L^{\top}L$, and where $\text{vec}\lp\cdot\rp$ is the vectorize operator that reshapes a matrix into a column vector. Solving directly for $L$ instead of $P$ enforces the constraint that $P$ must be symmetric positive semi-definite \cite[Thm. 7.2.7, p. 440]{horn2012matrix}. The norm-squared metric was chosen as it is the maximum likelihood estimator for i.i.d. Gaussian noise. The implicit formulation ensures the estimation is robust to measurement noise, and constructing estimation problems using an implicit formulation was successfully demonstrated in challenging environments such as estimating dynamics models with data from highly variable human behavior \cite{bonner2026pilot}. The LQR optimal feedback control follows from \eqref{eq:NLP} with $u=-Kx$.

\subsection{A Note On Data Collection}
As is the case with any method based solely on sampled data, care must be taken with respect to how the data is collected\footnote{As a concrete example, see the model-identification method known as Dynamic Mode Decomposition\cite{proctor2016dynamic}, and the corresponding note in \cite[p. 395]{brunton2022data}.}. Qualitatively, control inputs should produce a trajectory that adequately spans the $n$-dimensional state space of the original open-loop system. Typically, the assumption of \emph{persistent excitation}\cite{willems2005note,van2020willems} is made with respect to the input. Formal quantification of the diversity of the data, while an active area of current research, is beyond the scope of this current article and is left to future work.

 \subsection{A Note on Non-Zero Equilibrium Points}
The structure of the NLP problem in \eqref{eq:NLP} assumes that the unknown model is being perturbed about the equilibrium point $\{\xeq, \ueq \} = \{\mathbf{0}, \mathbf{0}\}$. This criteria can be generalized for a nonzero equilibrium state $\xeq\in \Rn$ and control $\ueq\in\Rm$. The method is generalized to the perturbation space with
\begin{equation}
    \begin{aligned}
        \delta x = x - \xeq, \\
        \delta u = u - \ueq.
    \end{aligned}
    \label{eq:perturbed state and control}
\end{equation}
See, for example, \cite[Fig. 2.1]{hespanha2018linear} for a visual depiction of these quantities. Likewise, we denote by $\delta x_k$ and $\delta u_k$ as the perturbed state at time $t_k$. By applying \eqref{eq:discrete equality} to \eqref{eq:perturbed state and control}, the equilibrium pair $\{\xeq, \ueq \}$ are jointly solved, along with $K$, from the following optimization problem:
\begin{equation}
\begin{cases}
\underset{L,K,\delta X, \xeq,\ueq}{\min} & \left\Vert \text{vec}\left(\delta Y\right)-\text{vec}\left(\delta X\right)\right\Vert _{2}^{2}\\
\\
\underset{\forall k=0,\ldots,N-1}{\text{subject to}} & \frac{2}{\Delta t}\lp\dxk^{\top}L^{\top}L\dxkp-\dxk^{\top}L^{\top}L\dxk\rp\\
 & -\dxk^{\top}K^\top R K\dxk\\
 & +\dxk^{\top}Q\dxk-2\dxk^{\top}K^\top R\lp u_k-\ueq \rp=0,
\end{cases}\label{eq:NLP_trim}
\end{equation}
where $\delta Y = Y - \ueq$ and $\delta X:=\left[\delta x_{0},\ldots,\delta x_{N}\right]$. The control is given as $u=-K\lp x-\xeq \rp+\ueq$.

\section{Connection To Q-Learning}\label{sec:Q-learning}
Many existing methods attempting data-driven approaches to find
the controller without knowledge of the system matrices, $A$ and
$B$, do so under the so-called reinforcement learning paradigm. We
will show here that the method presented above can be equivalently
seen as a variant of this framework, though our method importantly avoids policy iteration for computation. We note that while the implicit
dynamic constraints given in \eqref{eq:cont. time equality constraint} are mathematically equivalent in the
continuous setting to what follows, it is observed that \eqref{eq:discrete equality} has significantly
greater numeric stability when implemented on sampled-data.

We begin by taking a Q-learning \cite{watkins1989learning,watkins1992q} point of view to
the above problem, and define a Q-function for a continuous-time case \cite{luo2014q} with
\[
Q\lp t;x \lp \cdot \rp,u\lp\cdot\rp\rp:=\int_{t}^{\infty}x\lp\tau\rp^{\top}Mx\lp\tau\rp+u\lp\tau\rp^{\top}Ru\lp\tau\rp d\tau.
\]
Note that the Q-function is simply an evaluation of the cost functional,
$J$ from \eqref{eq:infinite cost funct}, evaluated on an arbitrary trajectory, starting at time
$t$. This implies the following semi-group property along a trajectory
\begin{align}
Q\lp t;x\lp \cdot\rp,u\lp\cdot\rp\rp = & Q\lp t+T;x\lp \cdot \rp,u\lp\cdot\rp\rp\nonumber \\
 +&\int_{t}^{t+T}x\lp\tau\rp^{\top}Mx\lp\tau\rp+u\lp\tau\rp^{\top}Ru\lp\tau\rp.\label{eq:Q semi-group}
\end{align}
Trivially, the optimal $Q$ function is equivalent to the value function,
i.e.
\[
Q^{*}\lp t\rp :=\underset{u\lp\cdot\rp}{\min}\,Q\lp t;x\lp \cdot \rp,u\lp\cdot\rp\rp=V\lp x\lp t \rp \rp.
\]
Proceeding with an \emph{advantage learning} framework \cite{baird_reinforcement_1994}, we
define the advantage function as
\begin{equation}
{\mathbb{A}}\lp t;x \lp \cdot \rp ,u\lp\cdot\rp\rp:=Q\lp t;x(\cdot),u\lp\cdot\rp\rp-V\lp x(t)\rp.\label{eq:Advantage def}
\end{equation}

The following lemma introduces one such advantage function that will be useful in our analysis.
\begin{lemma}\label{lem:advantage}
Suppose $P$ satisfies the algebraic Riccati equation. Then the following advantage function satisfies \eqref{eq:Q semi-group} $:$
\begin{equation}
{\mathbb{A}}\lp t;x \lp \cdot \rp,u\lp\cdot\rp\rp=\int_{t}^{\infty}\left\Vert u\lp\tau\rp+R^{-1}B^{\top}Px\lp\tau\rp\right\Vert _{R}^{2}d\tau,\label{eq:Advantage formula}
\end{equation}
where $\|\cdot\|_{R}=\sqrt{\langle\cdot,R\cdot\rangle}$ denotes the
norm induced by the symmetric positive definite matrix $R$.
\end{lemma}

The proof is given in the appendix. It is worth noting there is an additional body of literature under the name integral reinforcement learning \cite{lee2012integral,lee2014integral} when these advantage functions contain integral expressions.

Lemma \ref{lem:advantage} allows us to write $\eqref{eq:Q semi-group}$ as
\begin{align*}
 & V\lp x\lp t\rp\rp+\int_{t}^{t+T}\left\Vert u\lp\tau\rp+R^{-1}B^{\top}Px\lp\tau\rp\right\Vert _{R}^{2}d\tau\\
 & =V\lp x\lp t+T\rp\rp\\
 & +\int_{t}^{t+T}x\lp\tau\rp^{\top}Mx\lp\tau\rp+u\lp\tau\rp^{\top}Ru\lp\tau\rp d\tau.
\end{align*}
Referencing \eqref{eq:optimal value funct} and using the notation defined in \eqref{eq:notation} we arrive at the following equality:
\begin{align*}
 & \lp x_{k+1}^{\top}Px_{k+1}-x_{k}^{\top}Px_{k}\rp\\
 =&\int_{t}^{t+T}\left\Vert u\lp\tau\rp+R^{-1}B^{\top}Px\lp\tau\rp\right\Vert _{R}^{2}d\tau\\
  &-\int_{t}^{t+T}x\lp\tau\rp^{\top}Mx\lp\tau\rp+u\lp\tau\rp^{\top}Ru\lp\tau\rp d\tau.
\end{align*}
If Euler's method is used to discretize the integral
term above, then after some algebra we arrive at the expression
\begin{align}
\frac{1}{\Delta t}\lp x_{k+1}^{\top}Px_{k+1}-x_{k}^{\top}Px_{k}\rp & =x_{k}^{\top}K^\top \R Kx_{k}\nonumber \\
 & -x_{k}^{\top}\Q x_{k}+2x{}_{k}^{\top}K^\top \R u_{k}.\label{eq:alt discrete equality}
\end{align}
Noting that the left hand side of the above equation is a simple (forward) finite difference approximation of $\dot{V}\lp x_k \rp$, we conclude that the advantage function in \eqref{eq:Advantage formula} gives rise to a Q-function that is equivalent to a first-order approximation of the necessary condition \eqref{eq:Vdot optimality condition} given in Theorem \ref{thm: main thm}. 

\section{Model-Free Reference Tracking Control} \label{sec: ref tracking}
Thus far, the model-free controller utilizes a state feedback structure which designs a controller of the form $u =-Kx$ to stabilize a system's transient response, i.e. drive the state vector to zero. A more commonly encountered problem is the need to find a controller that tracks an input reference command. Denote by $r\in\mathbb{R}^{q}$ the desired reference and we
define the tracking problem by a constant $H\in\mathbb{R}^{n\times q}$ such that
\[
x-Hr=0\implies r=y,
\]
where $y=Cx$ is the output. Note that $H$ is not necessarily unique and is commonly constructed to have $r$ track specific states. In that case, each column of $H$ would consist of zeros everywhere except for a one in the $i$-th row corresponding to the $i$-th state vector element that is to be tracked. Alternatively, a general construction of $H$ can be found from
\[
H=C^{\top}\left(CC^{\top}\right)^{-1}.
\]
The tracking problem has a controller with structure \cite{franklinFeedbackControlDynamic2018}
\[
u=-K\left(x-Hr\right)+Fr,
\]
where $F\in\mathbb{R}^{m\times q}$ is a feedforward gain matrix required to eliminate steady-state error when $r=y$. The reference tracking control structure is illustrated in Figure \ref{fig:ref_track_blkdiag}. The model-free tracking problem is to simultaneously compute the stabilizing LQR feedback gain matrix $K$ and the feedforward gain matrix $F$ without knowledge of the system dynamics.

Define the following variables:
\begin{equation}
    \begin{aligned}
        \tx(t) & :=  x(t) - Hr, \\
        \tu(t) & :=  u(t) - Fr.
    \end{aligned}
    \label{eq:ref_deltas}
\end{equation}
Note that by definition $\tx \rightarrow 0$ and $\tu \rightarrow 0$ when $r=y$. Therefore, the gain $K$ is sought such that it optimizes
\begin{align}
 J\lp  \tilde     x_{0}, \tilde u\lp\cdot\rp\rp &=\int_{0}^{\infty} \tx\lp t\rp^{\top}\Q \tx\lp t\rp+ \tu\lp t\rp^{\top}\R \tu\lp t\rp dt.\label{eq:infinite cost funct ref}
\end{align}

Suppose the trajectory pair $\left[0,T\right]\ni t \mapsto\left\{ x(t), u(t)\right\}$ is an arbitrary solution to $\lp \ref{eq:LTIsys} \rp$ under the same assumptions as Theorem \ref{thm: main thm}, and let $\hr\neq 0$ be an arbitrary fixed constant reference. Then Theorem \ref{thm: main thm} is applied to $\lp \ref{eq:infinite cost funct ref} \rp$ and results in
\begin{equation}
    \begin{split}
            \dot{V}\lp \tx \lp t\rp\rp &  =\tx\lp t\rp^{\top}K^{\top}\R K\tx\lp t\rp\\ 
            &-\tx\lp t\rp^{\top}\Q\tx\lp t\rp+2\tx\lp t\rp^{\top}K^{\top}\R \lp u\lp t\rp-F\hr \rp,
    \end{split}
    \label{eq:Vdot_ref_track}
\end{equation}
where $\tx\lp t \rp$ is defined as in \eqref{eq:ref_deltas} using the solution $x\lp t \rp$.

When the same finite different approximation  detailed in Section \ref{sec:data driven control} is applied to \eqref{eq:Vdot_ref_track}, the following equality holds for all $k=0,\dots,N-1$:
\begin{align}
\frac{2}{\Delta t} & \lp \tx_{k}^{\top}P\tx_{k+1}-\tx_{k}^{\top}P\tx_{k}\rp  = \nonumber\\ 
& \tx_{k}^{\top}K^\top \R K\tx_{k}-\tx_{k}^{\top}\Q\tx_{k}+2\tx_{k}^{\top}K^\top \R \lp \uk-F\hr \rp.\label{eq:ref_discrete_vdot}
\end{align}
Given a set of observations $\{u_k,y_k\}_{k=0,\dots,N}$, and denoting by $Y$ as the array defined in \eqref{eq:bigY}, then for a fixed arbitrary reference $\hr\neq0$, the model-free reference tracking problem finds the triple $\lp P,K,F\rp$ that satisfies \eqref{eq:Vdot_ref_track}, using the constrained optimization problem:
\begin{equation}
    \begin{cases}
        \underset{L,K,F,\tilde X}{\min} & \left\Vert \text{vec}\left(\tilde Y\right)-\text{vec}\left(\tilde X\right)\right\Vert _{2}^{2}\\
        \\
        \underset{\forall k=0,\ldots,N-1}{\text{subject to}} & \frac{2}{\Delta t}\lp\tx_{k}^{\top}L^{\top}L\tx_{k+1}-\tx_{k}L^{\top}L\tx_{k}\rp\\
        & -\tx_{k}^{\top}K^\top \R K\tx_{k}\\
        & +\tx_{k}^{\top}\Q\tx_{k}-2\tx_{k}^{\top}K^\top \R\lp \uk-F\hr \rp=0,
    \end{cases}\label{eq:NLP_ref}
\end{equation}
where $\tilde Y = Y - H\hr$ and $\tilde X:=\left[\tx_{0},\ldots,\tx_{N}\right]$. As before, $P$ is found from $P = L^\top L$.

\begin{figure}[ht]
    \centerline{\includegraphics[width=\linewidth]{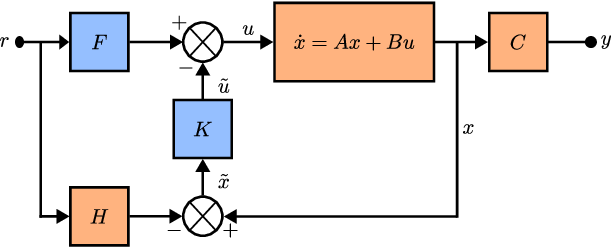}}
    \caption{Solving the reference tracking NLP \eqref{eq:NLP_ref} recovers the LQR feedback gain $K$ and feedforward gain $F$, both shown in blue. $H$ is a given matrix that defines the tracking  such that $y=r$. The reference tracking controller is implemented by using the controller $u=-K\left(x-Hr\right)+Fr$.}
    \label{fig:ref_track_blkdiag}
\end{figure}

\section{Mixed-Model Controller Design}\label{sec:mixed model optimality}

The goal of the mixed-model problem is to find the optimal LQR controller for systems in which the part of the dynamics are known, but the rest remain unknown. For instance an aircraft where the flight control surfaces are manipulated by an actuator with known dynamic response. In such scenarios, a mixed-model can be constructed that integrates the known dynamics of the actuators with the remaining system of unknown dynamics. The mixed-models being considered here are restricted to systems where the known dynamics operate as outer-loop inputs (i.e. control actuators) into the unknown dynamics.

Denote by $\gamma\in\Rp$ the state vector of the actuator system that behaves according to dynamics
\begin{equation}
    \begin{cases}
        \dot{\gamma}(t) = \hat A \gamma(t) + \hat B \hu(t) \\
        \gamma\lp 0\rp = \gamma_0, \\        
    \end{cases} \label{eq:mixed_LTI_known}
\end{equation}
almost everywhere $t\in[0,T]$. The system matrices $\hat A$ and $\hat B$ are assumed to be known. The unknown system in \eqref{eq:LTIsys} takes as input $\gamma$ and is written as:
\begin{equation}
        \dot{x}(t) =  Ax(t) + B\gamma(t). \\     
 \label{eq:mixed_LTI_unknown}
\end{equation}
Define $z = [x;\gamma] \in\mathbb{R}^{n+p}$ as the state vector of the joint system comprising of both unknown and known dynamics. Then the dynamics of \eqref{eq:mixed_LTI_unknown} is written as
\begin{align}
    \dot z(t) = \begin{bmatrix}
        \dot x(t) \\ \dot \gamma(t)
    \end{bmatrix}
    & = \begin{bmatrix}
        A & B \\ 0 & \hat A
    \end{bmatrix}z(t) + \begin{bmatrix}
        0 \\ \hat B
    \end{bmatrix}u(t) \nonumber\\
    & = \tilde A z(t) + \tilde B u(t),
    \label{eq:zdot}
\end{align}
with $z(0) = z_0 = [x_0;\gamma_0]$.

The objective of the mixed-model LQR problem is to find the gain $K$ such that it optimizes
\begin{align}
J\lp  z_{0}, \hu\lp\cdot\rp\rp &=\int_{0}^{\infty} z\lp t\rp^{\top}\Q z\lp t\rp+ \hu\lp t\rp^{\top}\R \hu\lp t\rp dt.\label{eq:infinite cost funct mixed}
\end{align}
Given a trajectory pair $[0,T]\ni t\mapsto \{z(t),\hu(t)\}$ that is a solution to \eqref{eq:zdot}, then Theorem \ref{thm: main thm} applies to \eqref{eq:infinite cost funct mixed} and gives
\begin{equation}
    \begin{split}
            \dot{V}\lp z \lp t\rp\rp &  =z\lp t\rp^{\top}P\tilde BR^{-1}\tilde B^\top Pz\lp t\rp\\ 
            &-z\lp t\rp^{\top}Qz\lp t\rp+2z\lp t\rp^{\top}P\tilde B \hu.
    \end{split}
    \label{eq:Vdot_mixed}
\end{equation}
Since $\tilde B$ is given, \eqref{eq:Vdot_mixed} is written without replacing $P\tilde B$ for $K^\top R$ and only $P$ is left to be found. When the same finite different approximation  detailed in Section \ref{sec:data driven control} is applied to \eqref{eq:Vdot_mixed}, the following equality holds for all $k=0,\dots,N-1$:
\begin{align}
\frac{2}{\Delta t} & \lp z_{k}^{\top}Pz_{k+1}-z_{k}^{\top}Pz_{k}\rp  = \nonumber \\ 
& z_{k}^{\top}P\tilde BR^{-1}\tilde B^\top Pz_{k}-z_{k}^{\top}Qz_{k}+2z_{k}^{\top}P\tilde B \hu_k.\label{eq:mixed_discrete_vdot}
\end{align}

Given a set of observations $\{\hu_k,y_k\}_{k=0,\dots,N}$ and known system dynamics $\hat A$ and $\hat B$, the mixed-model LQR problem finds $P=L^\top L$ using the constrained optimization problem:
\begin{equation}
\begin{cases}
\underset{L,Z}{\min} & \left\Vert \text{vec}\lp Y\rp-\text{vec}\lp Z\rp\right\Vert _{2}^{2}\\
\\
\underset{\forall k=0,\ldots,N-1}{\text{subject to}} & \frac{2}{\Delta t}\lp\zk^{\top}L^{\top}L\zkp-\zk^{\top}L^{\top}L\zk\rp\\
 & -\zk^{\top}L^{\top}L\tilde B\R^{-1}\tilde B^{\top}L^{\top}L\zk\\
 & +\zk^{\top}\Q\zk-2\zk^{\top}L^{\top}L\tilde B\hu_k=0 \\
 \\
 & \frac{1}{\Delta t}\lp\gamma_{k+1} - \gamma_k\rp = \hat A \gamma_k + \hat B \hu_k.
\end{cases}\label{eq:NLP-mixed}
\end{equation}
The optimization variable $Z = [X; \Gamma]$ is the stacked array of unknown states $X$ and known states $\Gamma$, where $\dot\gamma(t)$ is enforced by including an additional constraint which is based on the same finite difference approximation as \eqref{eq:finte diff in V} applied to \eqref{eq:mixed_LTI_known}. The mixed-model LQR controller is determined from $K = R^{-1}\tilde B^\top P$.

\subsection{Mixed-Model with Reference Tracking}

The mixed-model LQR problem can be updated to accommodate reference tracking in a process analogous to that of Section \ref{sec: ref tracking}. Defining the variable $\tilde z(t) = z(t) - Hr$ the constraint \eqref{eq:mixed_discrete_vdot} is updated to
\begin{align}
\frac{2}{\Delta t} & \lp \tilde z_{k}^{\top}P\tilde z_{k+1}-\tilde z_{k}^{\top}P\tilde z_{k}\rp  = \nonumber\\ 
& \tilde z_{k}^{\top}P\tilde BR^{-1}\tilde B^\top P\tilde z_{k}-z_{k}^{\top}Q\tilde z_{k}+2\tilde z_{k}^{\top}P\tilde B \lp\hu_k-F\hat r\rp,\label{eq:mixed_track_discrete_vdot}
\end{align}
where $\hat r$ is a fixed arbitrary reference and $F$ is the desired feedforward gain. The mixed-model reference tracking controller is constructed by $\hu(t) = -R^{-1}\tilde B^\top P \lp x - Hr\rp + Fr$.

\section{Results}\label{sec: results}

Two results cases will be presented. First, the solution methodology will be verified with a known linear system using the workflow first detailed in \cite{bowerfind2025model} and presented in Figure \ref{fig:flowchart}. Data is generated from a linear system and a model-free reference tracking controller is synthesized using the NLP \eqref{eq:NLP_ref}. The state feedback gain and reference tracking feedforward gain which are computed from the model-free method are denoted as $\hat K$ and  $\hat F$, respectively. A comparison is offered between the model-free controller and the classical LQR solution, computed when $A$ and $B$ are given and the ARE can be solved directly. The true LQR gain matrix is denoted as $K^*$. The true feedforward gain matrix, computed by the methods in \cite{franklinFeedbackControlDynamic2018} is denoted as $F^*$. Additionally, the same verification process will be accomplished for a system with a known mixed-model structure. That is, known actuator models will be introduced into the linear system and the mixed-model reference tracking NLP \eqref{eq:NLP-mixed} will be solved. 

\begin{figure}[ht]
    \centerline{\includegraphics[width=\linewidth]{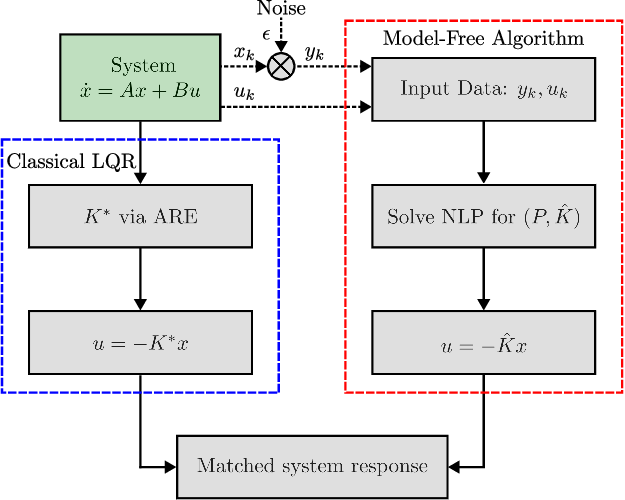}}
    \caption{The testing paradigm for Example~\ref{subsec: known_linear_results}}
    \label{fig:flowchart}
\end{figure}

Second, the model-free controller is applied to a real-world flight control problem. A model-free controller is designed to stabilize the roll angle on a subscale aircraft with unknown, nonlinear dynamics. Specifications of the test aircraft and flight instrumentation are provided along with a brief discussion about flight mechanics, which is used to facilitate the controller design process. The objectives and results of the data collection test flight are presented, and validate the model-free control alogorithm.

All data processing and optimization tasks were implemented using MATLAB R2025b.  For the optimization problems, the NLP code IPOPT \cite{wachter2006implementation} was used, where
the constraint Jacobian was computed via automatic differentiation using the methods of \cite{andersson2019casadi}. The NLP was initialized by setting $X = Y\in\mathbb{R}^{n\times N} $, $L = I \in \Rnn$, $K = {\mathbf{1}} \in \Rmn$ and $F = \mathbf{1}\in\mathbb{R}^{n\times q}$.
 
\subsection{Known Linear System} \label{subsec: known_linear_results}

The first test case is the linearized lateral-directional equations of motion for an aircraft \cite[Sec. 10.3, p. 572]{schmidtModernFlightDynamics2012}. The coefficients for this fourth-order model are from \cite[App. B, p. 819]{schmidtModernFlightDynamics2012} and represent a McDonnell Douglass A-4D Skyhawk traveling at 577 ft/s (Mach 0.6) at an altitude of 35,000 ft. The lateral state vector is defined as $x = \begin{bmatrix} \beta & \phi & p & r\end{bmatrix}^\top$, which corresponds to the aircraft sideslip angle, roll angle, and body-centric roll and  yaw rates. The input is $u = \begin{bmatrix} \delta_a & \delta_r\end{bmatrix}^\top$, which models the aileron and rudder deflections. The linearized lateral dynamics, in terms of the aerodynamic coefficients provided in Appendix \ref{app: A4_data}, are

\begin{align} \label{eq:lateral_dynamics}
    \dot x = & \begin{bmatrix}
        \frac{Y_\beta}{U_0} & \frac{g}{U_0} & \frac{Y_p}{U_0} & \lp \frac{Y_r}{U_0} - 1\rp \\
        0 & 0 & 1 & 0 \\
        L_\beta  & 0 & L_p  & L_r  \\
        N_\beta  & 0 & N_p  & N_r  \\
    \end{bmatrix} x +
     \begin{bmatrix}
         \frac{Y_{\delta_a}}{U_0} & \frac{Y_{\delta_r}}{U_0} \\
         0 & 0\\
         L_{\delta_a} & L_{\delta_r}\\
         N_{\delta_a} & N_{\delta_r}\\
     \end{bmatrix} u.
\end{align}

The controller is designed to track roll angle commands which implies $H = \begin{bmatrix}0 & 1& 0 &0\end{bmatrix}^\top$. A constant reference of $\hat r = 5^\circ$ was selected for use in \eqref{eq:NLP_ref}. Data was generated by perturbing the open-loop dynamics \eqref{eq:lateral_dynamics} with simultaneous linear-frequency chirp signals on the two input channels, with the parameters for each signal provided in Table \ref{tab:A4_chirp_params} in Appendix \ref{app: linear_sys_data}. The state measurement and control input data $\{y_k,u_k\}$, provided in Figure \ref{fig:A4_roll_data} in Appendix \ref{app: linear_sys_data}, were sampled at 40 Hz for $T=30$ seconds. The user-supplied LQR weight matrices were $\Q = \text{diag}\begin{bmatrix} 1&5&2&1\end{bmatrix}^\top$ and $\R = \text{diag}\begin{bmatrix} 1&1\end{bmatrix}^\top$. 

Supplying $\Q$, $\R$, $H$, $\hat r$, and the pair $\{ \yk,\uk\}$ as inputs into the model-free reference tracking optimization problem \eqref{eq:NLP_ref}, the optimal closed-loop model-free controller was computed as
\begin{align}
    \hat K = \begin{bmatrix} -1.9841 & 2.0619 & 1.3834 & 0.5234 \\
                             0.1442 & 0.8597 & 0.5376 & -0.9944\end{bmatrix},
    \label{eq:K_MF_linear}
\end{align}
with the feedforward gain found to be
\begin{align}
\hat F = \begin{bmatrix}
    0.0251 \\ -0.0522
\end{bmatrix}.
    \label{eq:Nbar_MF}
\end{align}

The known linear dynamics were used to compute the true LQR solution to compare performance between the model-free controller and the LQR controller. Using the same $\Q$ and $\R$, MATLAB's LQR solver was used to solve the ARE and recover the true LQR controller as
\begin{align}
    K^* = \begin{bmatrix} -1.9720 & 2.0435 & 1.3938 &  0.6144 \\
                               0.1676 & 0.8942 & 0.5429 & -1.0581\end{bmatrix}.
    \label{eq:K_lqr_linear}
\end{align}
The true feedforward gain
\begin{align}
F^* = \begin{bmatrix}
    0.0313 \\ -0.0584
    \end{bmatrix},
    \label{eq:Nbar_opt}
\end{align}
and is found from $F^* = N_u + K\lp N_x - H\rp$, where $N_u$ and $N_x$ come from the following relation \cite[Sec. 7.5.2, p. 1121]{franklinFeedbackControlDynamic2018}:
\begin{align}
    \begin{bmatrix}A&B\\C&D\end{bmatrix}\begin{bmatrix}N_x\\N_u\end{bmatrix} =\begin{bmatrix}0\\1\end{bmatrix}.
\end{align}
Comparing \eqref{eq:K_MF_linear} and \eqref{eq:K_lqr_linear} affirms nearly identical resulting controllers. Any numerical discrepancies are a result of the finite difference approximation used in \eqref{eq:finte diff in V}. The closed-loop responses of each controller were simulated to explore the impact of any numerical deviations on the controller performance. Figure \ref{fig:A4_roll_test} shows the response of the system to $\pm10^\circ$ roll angle reference commands and verifies the model-free controller behaves closely to the ground-truth controller.

\begin{figure}[ht]
\centerline{\includegraphics[width=\linewidth]{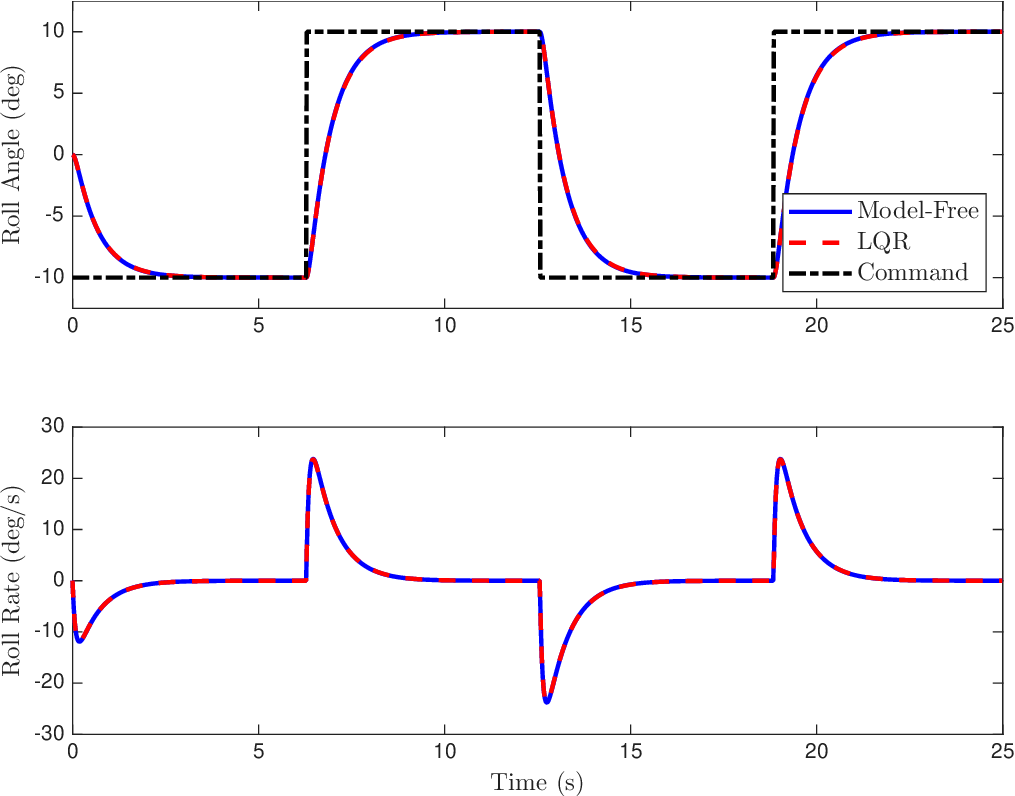}}
    \caption{The closed-loop simulation of A-4D lateral dynamics subject to roll angle tracking commands. The true LQR solution with $K^*$ is shown with the dashed red line. The model-free controller response using $\hat K$ is shown with the solid blue line. Roll angle tracking commands are indicated with the dashed black line.}
    \label{fig:A4_roll_test}
\end{figure}

Next, the mixed-model reference tracking will be validated. In this problem, known actuator models are augmented onto the lateral-direction dynamics to simulate servo dynamics inherent to the aileron and rudder assembly. The two actuator states, denoted as $\gamma = \begin{bmatrix} \alpha_{a} & \alpha_{r}\end{bmatrix}^\top$, which indicate the achieved state of the actuator, are appended onto the existing state vector so that  $z = \begin{bmatrix} x; \gamma\end{bmatrix}\in\mathbb{R}^6$. The control vector becomes the actuator command signal with $u = \begin{bmatrix} \delta_{a} & \delta_{r}\end{bmatrix}^\top$. A simple first-order model is chosen for each actuator, which gives the known state and control matrices as
\begin{align}
    \hat A & = \begin{bmatrix} -1\divslash \tau_a & 0 \\ 0 &-1\divslash \tau_r\end{bmatrix},
    \hat B & = \begin{bmatrix} 1\divslash \tau_a & 0\\0 & 1\divslash \tau_r\end{bmatrix}
\end{align}
where $\tau_a=0.05$s and $ \tau_r = 0.1s$. 

The mixed-model open loop dynamics were perturbed using a linear-frequency chirp signal, with parameters provided in Table \ref{tab:A4_chirp_params} in Appendix \ref{app: linear_sys_data}. Data was collected at 20 Hz for $T=30s$. See Figure \ref{fig:A4_roll_mixed_data} in Appendix \ref{app: linear_sys_data} for the state and control data pair $\{y_k,u_k\}$. The user defined LQR weight matrices were selected with $\Q = \text{diag}\begin{bmatrix} 1&5&2&1&10&10\end{bmatrix}^\top$ and $\R = \text{diag}\begin{bmatrix} 1&1\end{bmatrix}^\top$. 

Solving the  mixed-model reference tracking NLP \eqref{eq:NLP-mixed}, computes the optimal mixed-model closed-loop controller as
\begin{align}
    \hat K= \begin{bmatrix}-5.5911 & -0.1817 \\
                            2.0896 & 0.5785 \\
                            1.6535 & 0.2802 \\
                            1.7759 & -1.8825 \\
                            2.7411 & 0.0442 \\
                            0.0884 & 1.9989 \\\end{bmatrix}^\top.
    \label{eq:K_MF_mixed_linear}
\end{align}
The optimal mixed-model feedforward gain was computed to be
\begin{align}
    \hat F= \begin{bmatrix} 0.1047 \\
                            -0.0947 \\
                            \end{bmatrix}.
    \label{eq:N_MF_mixed_linear}
\end{align}

Similarly, MATLAB's LQR solver was used to compute the true LQR controller as
\begin{align}
    K^* = \begin{bmatrix}-5.4475 & 0.0135 \\
                                2.0904 & 0.6107 \\
                                1.6204 & 0.3081 \\
                                1.7052 & -1.8357 \\
                                2.5244 & 0.0199 \\
                                0.0398 & 2.5236 \\\end{bmatrix}^\top. 
    \label{eq:K_lqr_mixed_linear}
\end{align}
The true feedforward gain was calculated to be
\begin{align}
    F^* = \begin{bmatrix} 0.0877 \\
                          -0.1019 \\
                        \end{bmatrix}. 
    \label{eq:F_lqr_mixed_linear}
\end{align}

Numerically, the mixed-model controller \eqref{eq:K_MF_mixed_linear} and the true LQR controller \eqref{eq:K_lqr_mixed_linear} very closely related with the small discrepancies attributable to derivative approximations and inadequate open-loop perturbation bandwidth. The closed loop performance, presented in Figure \ref{fig:A4_roll_mixed_test}, is the best indicator of the validity of the mixed-model controller and exhibits excellent performance compared to the true LQR solution.

\begin{figure}[t]
\centerline{\includegraphics[width=\linewidth]{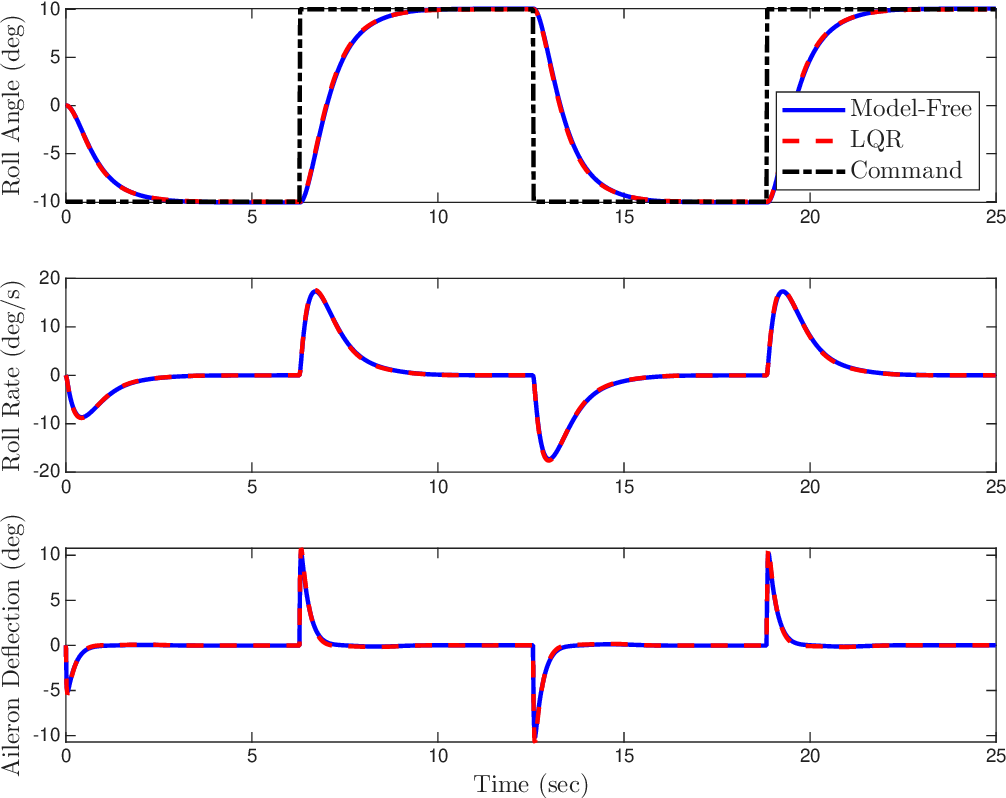}}
    \caption{The closed-loop simulation of A-4D lateral dynamics with  additional control surface actuator dynamics. The ground truth, generated by $\hat K$, is shown with the dashed red line. The mixed-model controller $K^*$ produces the response shown with the solid blue line. The desired roll angle tracking commands are shown with the dashed black line.}
    \label{fig:A4_roll_mixed_test}
\end{figure}

\subsection{Model-Free Controller Flight Test Program}
The capability to design a stabilizing controller using data collected directly from a real-world system with unknown, nonlinear dynamics will be demonstrated using a sub-scale aircraft model. This example showcases the versatility of the approach, especially on systems that are unknown, difficult to linearize, or have high dimensional state spaces.

\subsubsection{Test Aircraft Specifications} \label{sec: results-aircraft}
 Flight tests were conducted using a Freewing AL37 Airliner. This $1/19^{th}$ scale Boeing 737-inspired fuselage was manufactured using Expanded Polyolefin (EPO) foam and stiffened with carbon fiber wing spars. The aircraft utilizes conventional fixed-wing aerodynamic control surfaces, which include of a pair of ailerons mounted on the main wing to provide primary roll authority, an elevator fixed to the horizontal stabilizer to generate pitching moments, and a rudder on the vertical stabilizer to drive yaw dynamics. All aerodynamic control surfaces are actuated with  independent servos. Propulsion is achieved with dual $70mm$ brushless electric ducted fans (EDFs) integrated beneath each wing, driven by a single 6000 mAh battery and regulated by a 60A ESC for each EDF. 

\begin{figure}[ht]
    \centerline{\includegraphics[width=1\linewidth]{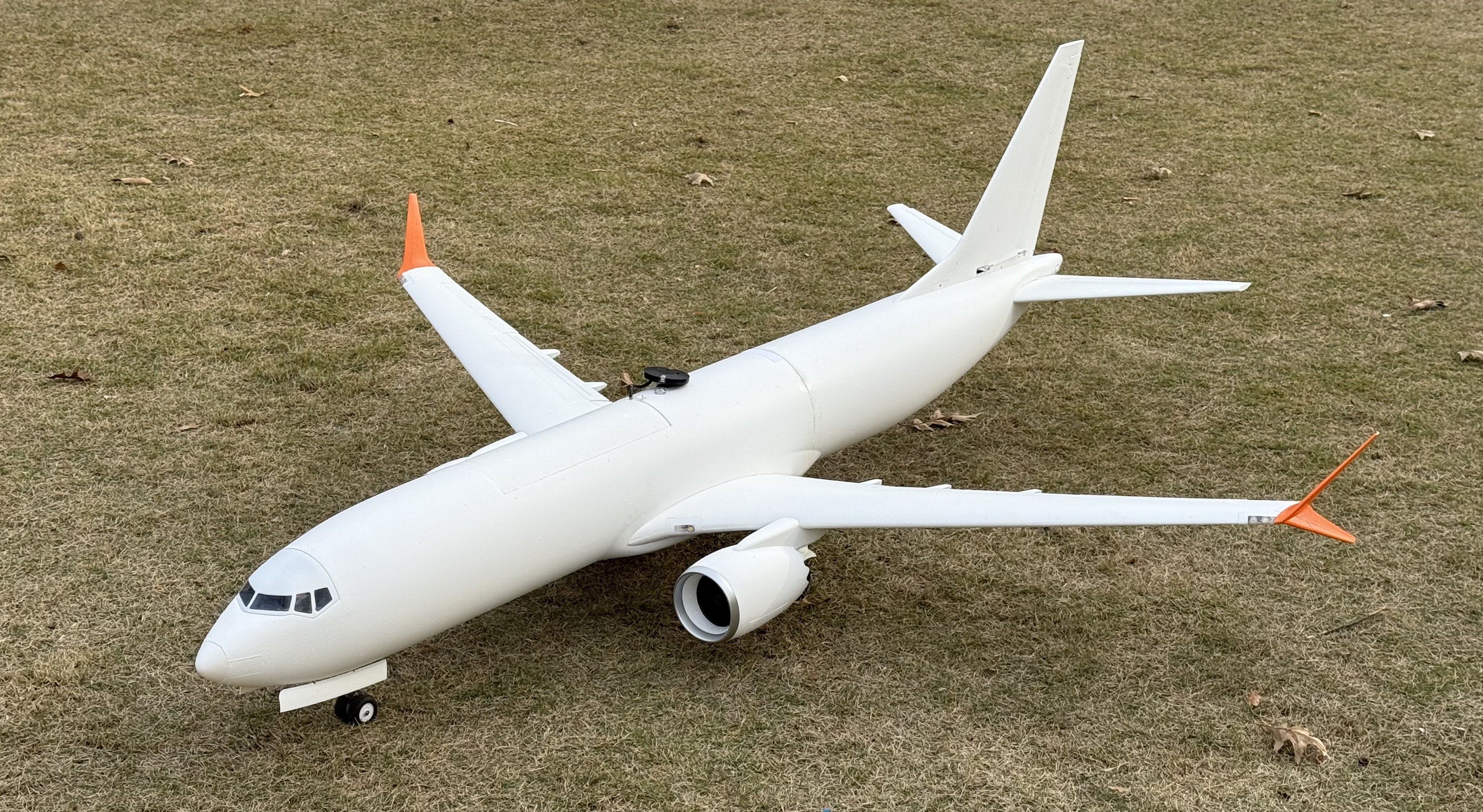}}
    \caption{The Freewing AL37 test aircraft.}
    \label{fig:AL37}
\end{figure}

\begin{table}[ht]
\centering
\caption{Freewing AL37 Specifications \cite{al37}}
\label{tab:AL37 specs}
    \begin{tabular}{|c|c|}\hline
        \multicolumn{2}{|c|}{Dimensions}\\
        \hline \hline
        Wingspan &  1830mm\\
        Length& 2000mm\\
        Wing area& 3600mm$^2$\\
        Weight (including battery and electronics)& 4542g\\
        \hline \hline
        
        \multicolumn{2}{|c|}{Componenents}\\
        \hline\hline
        Motor& 2952-2100Kv\\
        ESC& 60A with 8A UBEC \\
        Battery& 6S 22.2V 6000 mAh \\
        Ducted fan& 70mm, 12-blade \\ 
        Control surface servo& FS31092/3 9g servo\\
        \hline      
\end{tabular}
\end{table}

\subsubsection{Flight Computer and Sensor Package} \label{sec: results-sensors}
The fuselage interior houses a custom instrumentation and flight computer package. The Pixhawk Orange Cube+ provides all onboard computing and data logging. All open-sourced sensor and autopilot modules have been disabled on the Cube+. The Orange Cube+ is used as the platform for external sensor integration, data logging, model-free flight control, and connection with a radio receiver. The radio receiver is linked to a Futaba T6K 2.4GHz handheld transmitter, and is used to pilot the aircraft when collecting data for the model-free algorithm or when the resulting controller is disengaged during testing.  MATLAB support packages were used to compile and deploy the methods presented in this paper.

Two sensors provide the primary measurement capabilities onboard the aircraft. The first is a Life Performance Research LPMS-CURS3 inertial measurement unit (IMU) mounted internally, coincident with the vehicle center of gravity. This 9-Axis IMU outputs bias-calibrated measurements from a three-axis gyroscope, a three-axis accelerometer, and a magnetometer. Orientation estimates are provided by the IMU by internally fusing the output of the component measurements using a Kalman filter.  Additional specifications of the IMU are provided in Table \ref{tab: imu data} and further information can be found from the product datasheet \cite{lpms-curs3}.

The second sensor is a CubePilot Here 3+ GPS system with the antenna externally mounted on the top of the fuselage to provide line of sight with the GPS satellite constellation. The Here 3+ sensor is configured for an update rate of 8Hz and provides latitude and longitude positioning as well as derived measurements that include ground speed, track heading, and north-east-down (NED) velocity.

\begin{table}[ht]
\centering
\caption{LPMS-CURS3 IMU \cite{lpms-curs3} and Here 3+ GPS \cite{here3} Specifications as Configured For Tests That Appear in Section \ref{sec: results-flight test}.}
\label{tab: imu data}
\begin{tabular}{|c|c|}\hline
        \multicolumn{2}{|c|}{Accelerometer}\\
        \hline \hline
        \multirow{1}*{Measurement range} &  $\pm4\,g$\\
        \multirow{1}*{Linear acceleration sensitivity} & $0.122\,mg$/LSB \\
        \multirow{1}*{Acceleration noise density} & 60$\,\mu g$/$\sqrt{\text{Hz}}$ \\
        \hline \hline
        
        \multicolumn{2}{|c|}{Gyroscope}\\
        \hline \hline
        \multirow{1}*{Measurement range} &  $\pm500$ \textdegree/s\\
        \multirow{1}*{Angular rate sensitivity} & $17.50$ mdps/LSB \\
        \multirow{1}*{Rate noise density} & 5 mdps/$\sqrt{\text{Hz}}$ \\
        \hline \hline      

        \multicolumn{2}{|c|}{Magentometer}\\
        \hline \hline
         \multirow{1}*{Measurement range} & $\pm8$ Gauss \\
         Sensitivity & $12000$ LSB/Gauss \\
        \hline \hline

        \multicolumn{2}{|c|}{GPS}\\
        \hline \hline
         Positioning accuracy & 2.5 m CEP (3D fix)\\
         Update rate & $8$ Hz  \\
        \hline
\end{tabular}
\end{table}
 
\subsubsection{Aircraft Dynamics Fundamentals} \label{sec: results-dynamics}
Although the model-free design philosophy of this study eliminated any modeling requirements on the AL37 test vehicle, an understanding of the fundamental equations of motion is crucial to identify the relevant states that should be used to appropriately control the vehicle dynamics. A brief discussion of the six-degree of freedom rigid body equations of motion will be presented since they provide a theoretical framework for the following model-free control application.

The Flat-Earth inertial reference frame, denoted as $e$, extends in the north-east-down (NED) directions. The body-fixed axis, referred to as $b$, follows standard conventions with the \textit{x}-axis extending forward out the nose of the aircraft, \textit{y}-axis pointing out the right wing side, and the \textit{z}-axis downward completing the triad which is fixed at the center of gravity of the aircraft. The dynamics encompass twelve states: the inertial NED position, represented as $P_e = [\begin{matrix} p_n & p_e & p_d\end{matrix}]^\top$, the Euler orientation angles, expressed as $\Phi = [\begin{matrix} \phi & \theta & \psi \end{matrix}]^\top$, the body-fixed velocity, denoted as $V_b = [\begin{matrix} u_b & v_b & w_b\end{matrix}]^\top$, and the angular velocity of the body in the body frame, denoted as $\omega_{b} = [\begin{matrix} p & q & r \end{matrix}]^\top$. The equations of motion of the aircraft are given by \cite{stevens_aircraft_2015}:
\begin{equation}
\begin{cases}
            \begin{aligned}
                \dot P_e & = C_{be}(\Phi) V_b\\
                \dot V_b & = \frac{1}{m}F_b - \omega_{b} \times V_b\\
                \dot \Phi & = W\lp \Phi\rp \omega_{b}\\
            \dot \omega_{b} & = I_b^{-1} \left[ M_b - \omega_{b} \times \lp I_b \omega_{b} \rp \right],
            \end{aligned}
        \label{eq:6dof trans}
    \end{cases}
\end{equation}
where $C_{be}(\Phi)$ is the rotation matrix that maps body frame to earth frame and
\begin{equation}
    \begin{aligned}
        W (\Phi) = \begin{bmatrix}
                1 & \sin\phi\tan\theta & \cos\phi\tan\theta\\
                0 & \cos\phi & -\sin\phi\\
                0 & \sin\phi/\cos\theta & \cos\phi/\cos\theta
            \end{bmatrix}. 
    \end{aligned}
\end{equation}
The quantities $F_b$ and $M_b$ are the forces and moments applied to the airframe. For aircraft, these tend to be nonlinear functions of the state vector that include aerodynamic, gravitational and propulsive factors. Without knowing these complex effects, especially that of aerodynamics, it is impossible to fully construct \eqref{eq:6dof trans}, and therefore this example showcases the contribution of the proposed model-free algorithm.   

\subsection{Model-Free Controller Flight Test}\label{sec: results-flight test}
The model-free controller synthesis is demonstrated with three primary steps. First, conduct a flight test, with flight stabilization disabled, that records manual pilot control inputs and measures appropriate vehicle states using the IMU. Second, use the flight test data to solve the NLP \eqref{eq:NLP}. Lastly, perform a test flight to validate the stability and performance of the model-free LQR controller.

\subsubsection{TF1: Data Collection Test Flight}
The initial data collection test flight, designated as TF1, was executed with all onboard stabilization systems deactivated and manually operated by a human pilot. The objective of TF1 was to record IMU measurements and pilot commands as the plane was flown in an approximate oval racetrack pattern around the airfield.  All flight tests were conducted at the Mr. John M. Sistrunk Memorial Flying Field, which features a 700-foot grass runway and is located approximately 15 miles south of the Auburn University campus.

TF1 recorded IMU measurements at $50$Hz during a $120s$ test execution window. Data collected during the takeoff and landing flight phases was discarded since the vehicle was not operating in a cruise flight condition. The model-free LQR controller is similar to classical control design techniques in that data should be collected in a similar flight condition that the flight controller will later operate in.

\subsubsection{Model-Free Roll Controller Synthesis}
Designing a roll controller about trimmed, straight-and-level flight allows for simplification of the complete nonlinear dynamics by decoupling the lateral dynamics from the longitudinal dynamics. As such, only IMU measurements of the roll angle $\phi$ and roll rate $p=\dot \phi$ are required. These two measurements alone are sufficient to design a stabilizing roll controller, however, an additional integrator state, $\int\phi$, was added to regulate any nonzero steady-state error that may be caused by vehicle asymmetries. $\int\phi$ was obtained during postprocessing by numerically integrating the roll angle measurement signal.

The user-supplied LQR weight matrices were $\Q = \text{diag} \begin{bmatrix} 1.7748 & 0.3475 & 0.01\end{bmatrix}^\top$ and $\R = 10$.\footnote{These matrices were selected by comparing the magnitude of a hypothetical controller output under the data set $y$, and scaling $Q$ to avoid saturation (i.e. $u_{\min}\leq -KY\leq u_{\max}$).} Figure \ref{fig:roll_controller_data} presents the actual recorded measurements from TF1 that were used as the data supplied for the NLP solve. The three-state measurement vector was $Y = [\begin{matrix} \phi & \dot \phi & \int \phi\end{matrix}]^\top$ and the open-loop input perturbation $U = u_a$ was given by the raw pilot aileron commands. Solving problem \eqref{eq:NLP} computes the optimal model-free closed-loop controller
\begin{equation}
    \hat K = \begin{bmatrix} 0.4244 & 0.1264 & 0.0317\end{bmatrix}.
    \label{eq: K_MF_roll_controller}
\end{equation}

\begin{figure}[ht]
    \centerline{\includegraphics[width=\linewidth]{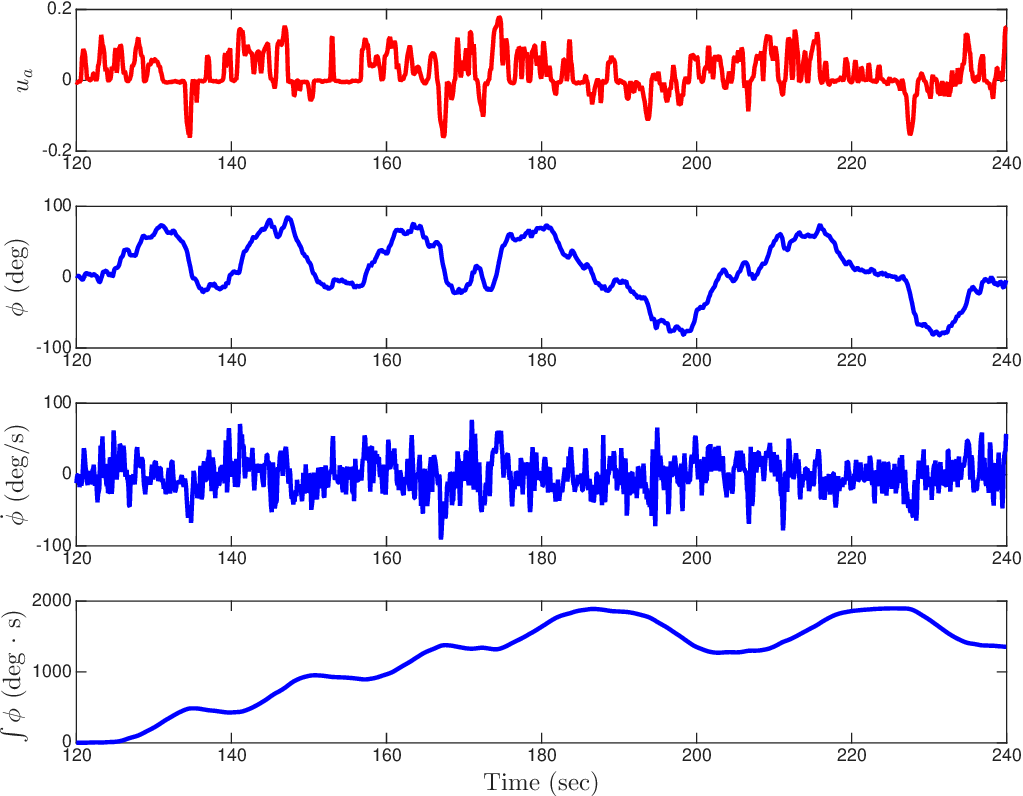}}
    \caption{The actual state measurements and pilot inputs from TF1 that were provided as inputs for the NLP problem \eqref{eq:NLP} used to compute $\hat K$. The top subplot displays the pilot aileron commands that were recorded directly from the handheld radio transmitter. The second and third subplots, of the roll angle and roll rate, depict IMU measurements. The integral of the roll angle measurement, shown in the fourth subplot, was computed via numerical integration of the roll measurement during postprocessing. }
    \label{fig:roll_controller_data}
\end{figure}

\subsubsection{TF2: Model-Free Controller Validation Flight}
The objective of the model-free controller evaluation test flight, designated as TF2, was to verify roll controller stability and expected performance. TF2 recorded IMU measurements at $50 Hz$ and GPS measurements at $8Hz$ during a $100s$ test execution window. Figure \ref{fig:TF2_track} depicts the TF2 ground track for which the model-free controller was engaged. The figure-eight flight path was planned so that the aircraft would maximize the time for which precise roll command tracking was required. Environmental conditions for TF2 were windy which induced the variations in groundspeed and altitude visible in Figure \ref{fig:TF2_track}.

\begin{figure}[ht]
\centerline{\includegraphics[width=\linewidth]{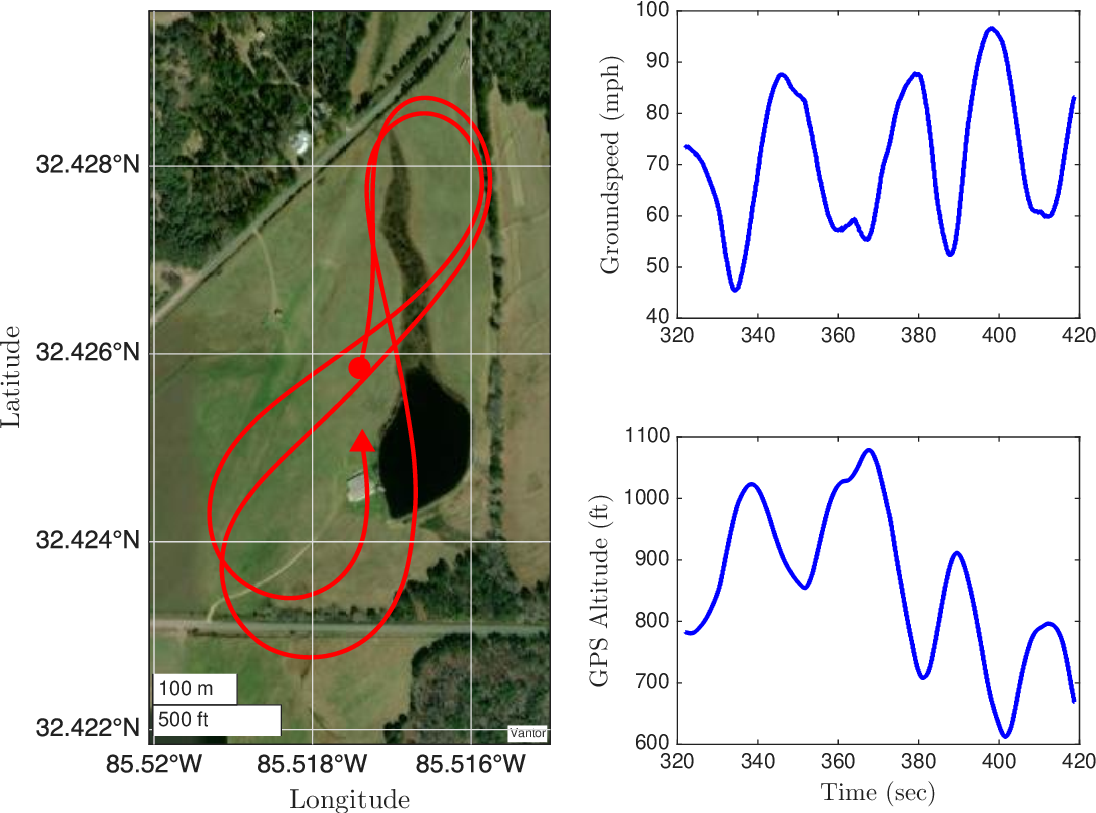}}
    \caption{GPS data from TF2 indicates that despite gusty flight operations, indicated by the significant groundspeed and GPS altitude variations, the pilot was able to achieve the desired figure-eight ground track aided by precise, intuitive performance of the model-free roll controller.}
    \label{fig:TF2_track}
\end{figure}

Figure \ref{fig:roll_controller_test} presents the roll angle tracking performance of the model-free controller \eqref{eq: K_MF_roll_controller}. The pilot roll commands, indicated in red, fluctuate as the pilot attempts to fly the figure-eight pattern in variable wind conditions. The true vehicle response, indicated in blue, clearly tracks the pilot commands. The controller successfully tracks a noisy input signal and achieves roll angle tracking with negligible steady-state error.

\begin{figure}[ht]
\centerline{\includegraphics[width=\linewidth]{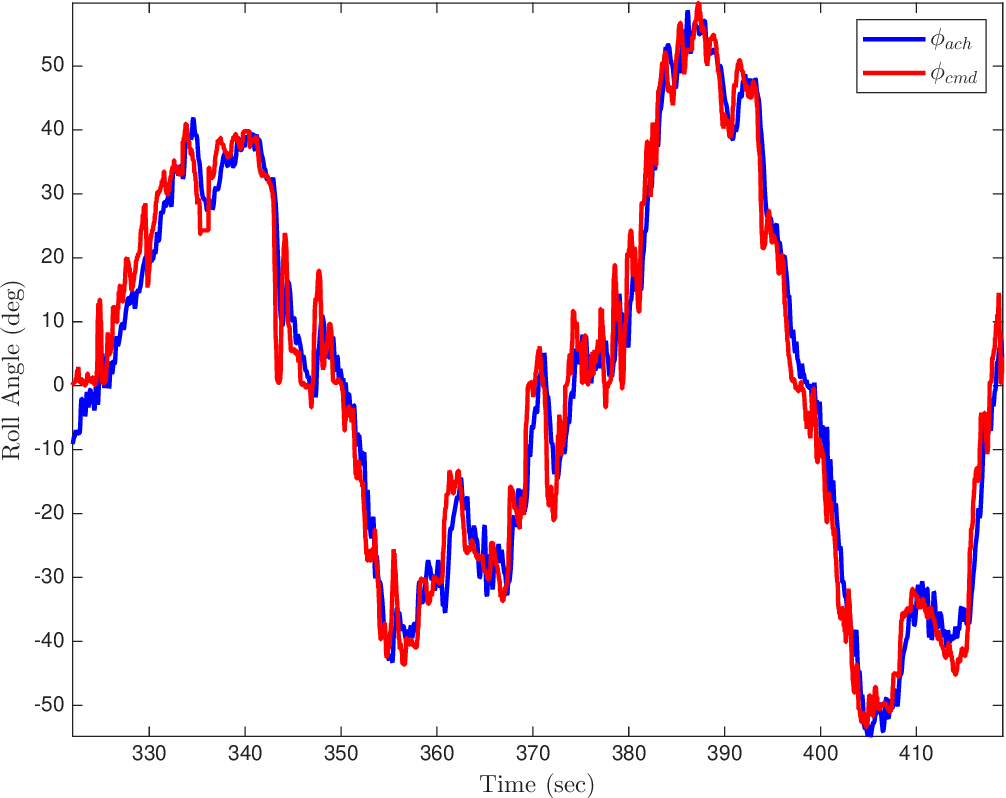}}
    \caption{The roll angle tracking performance of the model-free roll controller using $\hat K$. The pilot roll angle commands, shown with the red line, display a combination of high-frequency inputs, due to the windy flight condition, and large amplitude, low-frequency inputs used to fly the figure-eight trajectory. The achieved roll angle of the test aircraft, indicated with the blue line, very closely tracks the commanded roll angle indicating excellent tracking performance.}
    \label{fig:roll_controller_test}
\end{figure}

\section{Conclusion} \label{sec: conclusion}

A model-free approach to infinite horizon LQR synthesis directly from data collected of an unknown system is presented. A necessary condition for the value function is proved which gives a specific dynamic constraint along any arbitrary trajectory. An implicit NLP problem is constructed which is used to compute the optimal closed-loop LQR controller. The model-free approach is expanded to solve the model-free reference tracking LQR problem which solves for the optimal feedforward and feedback gains, simultaneously, by enforcing a similar dynamic constraint in an adapted NLP problem. Lastly, the model-free requirement is relaxed to include a broader class of systems for which part of the system dynamics is known, but the rest of the dynamics remain unknown.

Two simulation results are included to validate the approach relative to the ground-truth. Each example indicates that the performance of the model-free controller tracks performance of classical LQR controllers. 

Lastly, a real-world flight test program was conducted to design a roll controller for a subscale UAV with unknown dynamics. The data collection test flight, data-driven controller synthesis, and performance evaluation test flight are documented. Flight-tests confirm the excellent performance of the model-free controller during challenging flight operations.

Future avenues of study include the establishment of metrics for data collection that ensure reliable calculation of the controller. This would allow users to reduce the total amount of data required for accurate model-free controller synthesis. Additionally, the core algorithm presented here could be generalized to real-time model-free controller synthesis, which could form a robust adaptive controller that can stabilize the system under changing environmental and flight regime conditions.

\section*{Acknowledgement}
 
The authors thank Sean Gallagher, of Union Springs, AL, for skillfully piloting the AL37 aircraft for both data collection and later during flight controller validation. The authors also thank the Auburn Planesman R/C Flying Club for use of the airfield where all flight experiments occurred. The Auburn Planesman Club is a Charted Club (\#1580) of the Academy of Model Aeronautics.  

\appendices{}     

\section{Proofs From Section \ref{sec:Q-learning}}

Before providing a proof for Lemma \ref{lem:advantage}, we need the following supporting lemma, which is an adaptation of Lemma 6 in \cite{willems_least_1971}:

\begin{lemma}\label{lemma:2}
    Let $P$ be any symmetric solution of the algebraic Riccati equation \eqref{eq:ARE} for arbitrary $x$ and $u$ that satisfy \eqref{eq:LTIsys}. Then on the time interval $\tau\in\left[t,t+T\right]:$
\begin{align}
 & V\left(x\left(t\right)\right)\nonumber \\
 & +\int_{t}^{t+T}\left\Vert u\left(\tau\right)+R^{-1}B^{\top}Px\left(\tau\right)\right\Vert _{R}^{2}d\tau\nonumber \\
 & =V\left(x\left(t+T\right)\right)\nonumber \\
 & +\int_{t}^{t+T}x\left(\tau\right)^{\top}Mx\left(\tau\right)+u\left(\tau\right)^{\top}Ru\left(\tau\right)d\tau.\label{eq:lem for A}
\end{align}
\end{lemma}
\begin{IEEEproof}
    Differentiation by parts of \eqref{eq:optimal value funct} yields
    \begin{equation}
        \begin{aligned}
            V(x(t+T)) - & V(x(t)) = \int_t^{t+T} \frac{d}{d\tau} \lp x(\tau)^\top Px(\tau)\rp d\tau \\
            & = \int_t^{t+T} 2x(\tau)^\top P \lp Ax(\tau) + Bu(\tau)\rp d\tau. 
        \end{aligned}
        \label{eq:proof2e1}
    \end{equation}
    The first equality follows from the fundamental theorem of calculus, and the second equality follows from \eqref{eq:LTIsys}. Applying the Riccati equation \eqref{eq:ARE}, the integrand becomes
    \begin{equation}
        \begin{aligned}
             2x^\top& P (Ax + Bu) \\
             & = x^\top (A^\top P + PA)x + u^\top B^\top Px + x\top PBu\\
             & = x^\top (PBR^{-1}B^\top P - M)x + u^\top B^\top Px + x^\top PBu\\
             & = \|u + R^{-1}B^\top Px\|^2_R - u^{\top}Ru - x^\top Mx.
        \end{aligned}
        \label{eq:proof2e2}
    \end{equation}
    Evaluating \eqref{eq:proof2e2} at $x(\tau)$ and $u(\tau)$  and substituting into the integrand of \eqref{eq:proof2e1} produces the correct result.
\end{IEEEproof}

We now give the proof of Lemma \ref{lem:advantage}.

\begin{IEEEproof}
For simplicity we define the quantity
\[
W\left(\tau\right):=\left\Vert u\left(\tau\right)+R^{-1}B^{\top}Px\left(\tau\right)\right\Vert _{R}^{2}.
\]
Using this notation, \eqref{eq:Advantage formula} can be written as
\begin{align*}
{\mathbb{A}}\left(t;x \lp \cdot \rp,u\left(\cdot\right)\right) & :=\int_{t}^{\infty}W\left(\tau\right)d\tau\\
 & =\int_{t}^{t+T}W\left(\tau\right)d\tau+\int_{t+T}^{\infty}W\left(\tau\right)d\tau.
\end{align*}
This implies that
\begin{equation}
{\mathbb{A}}\left(t;x \lp \cdot \rp,u\left(\cdot\right)\right)=\int_{t}^{t+T}W\left(\tau\right)d\tau+{\mathbb{A}}\left(t+T;x \lp \cdot \rp,u\left(\cdot\right)\right).\label{eq:Advantage proof}
\end{equation}
By the definition given in \eqref{eq:Advantage def},
\[
{\mathbb{A}}\left(t;x,u\left(\cdot\right)\right):=Q\left(t,x\left(t\right),u\left(\cdot\right)\right)-V\left(x\left(t\right)\right).
\]
This implies that we can write \eqref{eq:Advantage proof} as
\begin{align*}
Q\left(t;x\left(\cdot\right),u\left(\cdot\right)\right)-V\left(x\left(t\right)\right)= & \int_{t}^{t+T}W\left(\tau\right)d\tau\\
 & +Q\left(t+T;x\left(\cdot\right),u\left(\cdot\right)\right)\\
 & -V\left(x\left(t+T\right)\right).
\end{align*}
Rearranging the above equation and substituting from Lemma \ref{lemma:2} yields
\begin{align*}
Q\left(t;x\left(\cdot\right),u\left(\cdot\right)\right) &=  Q\left(t+T;x\left(\cdot\right),u\left(\cdot\right)\right)\\
 & + \int_{t}^{t+T} \Big\{x\left(\tau\right)^{\top}Mx\left(\tau\right)\\
 &\quad+u\left(\tau\right)^{\top}Ru\left(\tau\right) \Big\} d\tau,
\end{align*}
and we have arrived at our result. 
\end{IEEEproof}

\section{A-4D Skyhawk  Aerodynamic Data} \label{app: A4_data}

Table \ref{tab:A4_coefs} presents the lateral-directional dimensional derivatives which are the coefficients for the model \eqref{eq:lateral_dynamics} which represent an McDonnell Douglas A-4D Skyhawk. This model is used in the known linear systems test case Section \ref{subsec: known_linear_results}.
\begin{table}[ht]
    \caption{A-4D Lateral-Directional Derivatives at 35,000ft and 0.6 Mach \cite[App. B, p. 819]{schmidtModernFlightDynamics2012}.}
    \label{tab:A4_coefs}
    \centering
    \begin{tabular}{|c c |c c|} 
        \hline
       $U_0$, (ft/s)  & 577 & $L_r$ & 0.475\\
       $Y_\beta$  & -60.386 & $L_{\delta_a}$ & 8.170\\
       $Y_p$  & 0 & $L_{\delta_r}$ & 4.168\\
       $Y_r$  & 0 & $N_\beta$ & 6.35\\
       $Y_{\delta_a}$  & -0.4783 & $N_p$ & -0.025138\\
       $Y_{\delta_r}$  & 10.459 &  $N_r$ & -0.2468\\
       $L_\beta$ & -17.557 & $N_{\delta_a}$  & 0.5703\\
       $L_p$ & -0.761 & $N_{\delta_r}$  & -3.16\\
        \hline      
    \end{tabular}
\end{table}

\section{Known Linear System Input Data} \label{app: linear_sys_data}
The data generated for the known linear system test cases in Section \ref{subsec: known_linear_results} uses a linear frequency chirp signal for each of the two control input dimensions. The linear frequency chirp was computed using the function $u_k = \psi \sin\lp2\pi\lp\frac{c}{2}t_k^2 + f_0t_k\rp\rp$, where $c=\lp f_1-f_0\rp/T$. Parameters for the input signal used for the model-free and mixed-model cases are provided in Table \ref{tab:A4_chirp_params}.
\begin{table}[ht]
    \caption{Control-Input Signal Parameters}
    \label{tab:A4_chirp_params}
    \centering
    \begin{tabular}{|c|c|c|c|c|}
    \hline
     & \multicolumn{2}{c|}{ Model-Free} & \multicolumn{2}{c|}{Mixed-Model} \\
    \hline
    Parameter & $\delta_a$ & $\delta_r$ & $\delta_a$ & $\delta_r$\\
    \hline
     $\psi$ & 0.01 & 0.015  & 0.01 & -0.015 \\
     $f_0$  & 0.1  & 0.1    & 0.2  & 0.1\\
     $f_1$  & 0.8  & 0.5    & 0.4  & 0.5 \\
     $c$    & 0.01 & 0.0133 & 0.0067 & 0.0267 \\
     \hline 
    \end{tabular}
\end{table}
The state data generated for the model-free case is provided in Figure \ref{fig:A4_roll_data}.
\begin{figure}[ht]
\centerline{\includegraphics[width=\linewidth]{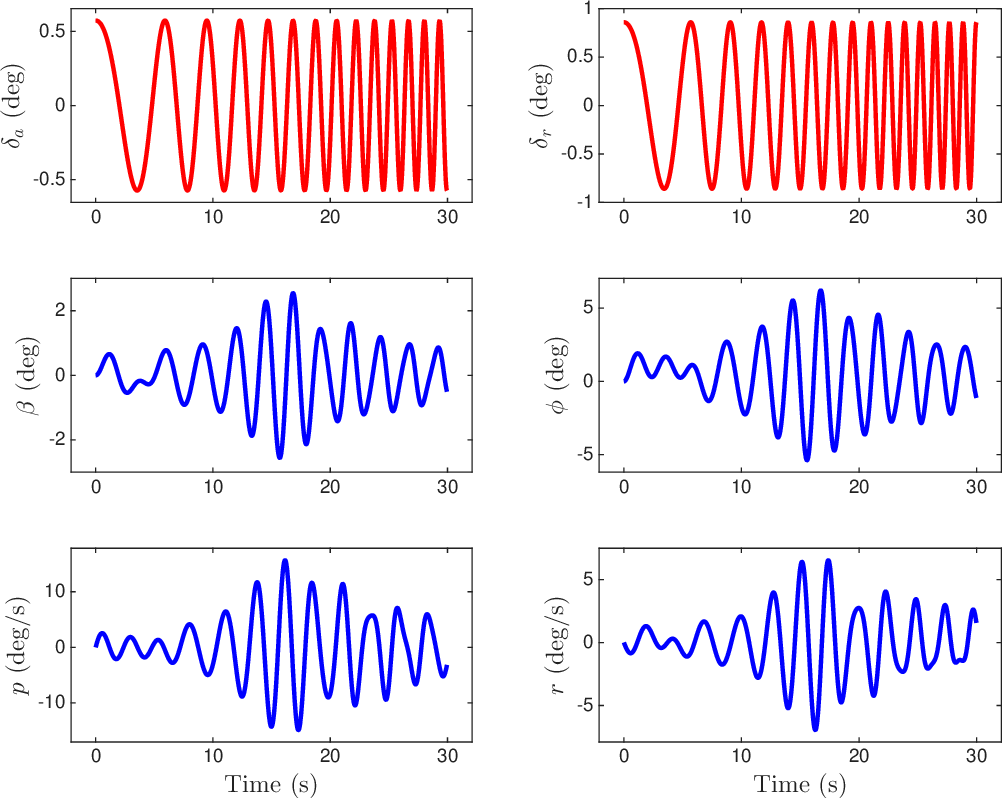}}
    \caption{The state input data $y_k$, shown in blue, and the control input signal $u_k$, shown in red, were supplied as inputs to the model-free NLP \eqref{eq:NLP_ref}. Data was sampled at $40Hz$ for $T=30s$.}
    \label{fig:A4_roll_data}
\end{figure}
The state data generated for the mixed-model case is provided in Figure \ref{fig:A4_roll_mixed_data}.
\begin{figure}[ht]
\centerline{\includegraphics[width=\linewidth]{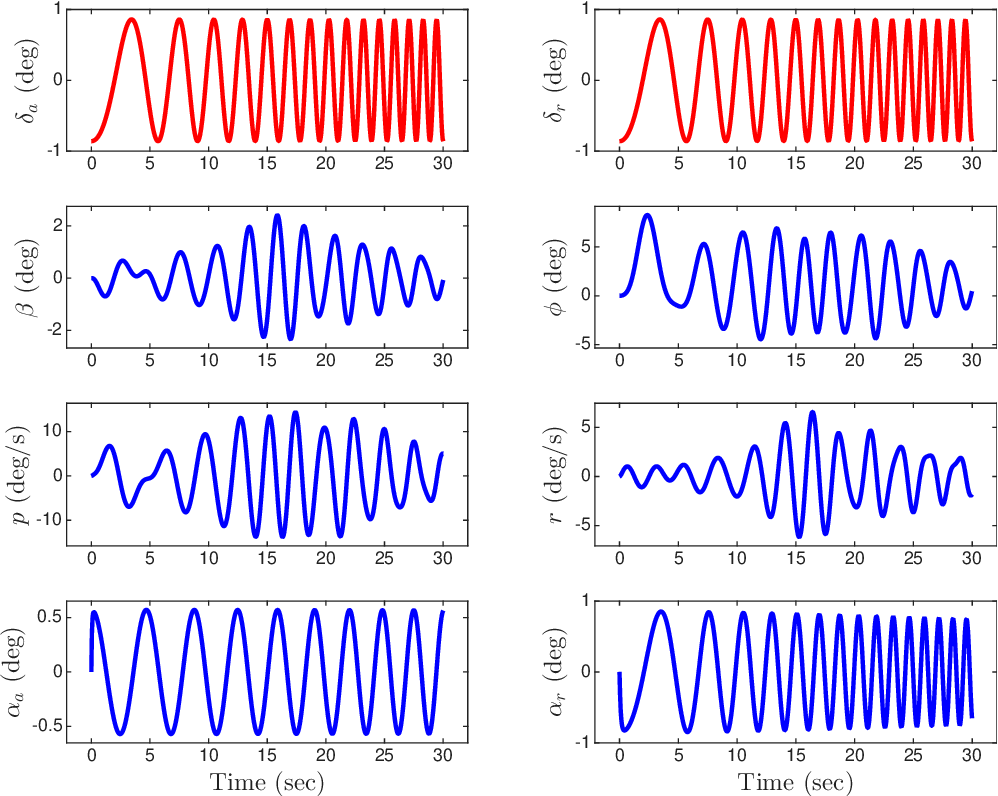}}
    \caption{The state input data $y_k$, shown in blue, and the control input signal $u_k$, shown in red, were supplied as inputs to the mixed-model NLP \eqref{eq:NLP-mixed}. Data was sampled at $20Hz$ for $T=30s$}
    \label{fig:A4_roll_mixed_data}
\end{figure}


\bibliographystyle{IEEEtran}
\bibliography{main}

\begin{thebibliography}{10}
\providecommand{\url}[1]{#1}
\csname url@samestyle\endcsname
\providecommand{\newblock}{\relax}
\providecommand{\bibinfo}[2]{#2}
\providecommand{\BIBentrySTDinterwordspacing}{\spaceskip=0pt\relax}
\providecommand{\BIBentryALTinterwordstretchfactor}{4}
\providecommand{\BIBentryALTinterwordspacing}{\spaceskip=\fontdimen2\font plus
\BIBentryALTinterwordstretchfactor\fontdimen3\font minus \fontdimen4\font\relax}
\providecommand{\BIBforeignlanguage}[2]{{%
\expandafter\ifx\csname l@#1\endcsname\relax
\typeout{** WARNING: IEEEtran.bst: No hyphenation pattern has been}%
\typeout{** loaded for the language `#1'. Using the pattern for}%
\typeout{** the default language instead.}%
\else
\language=\csname l@#1\endcsname
\fi
#2}}
\providecommand{\BIBdecl}{\relax}
\BIBdecl

\bibitem{safonov1977gain}
M.~Safonov and M.~Athans, ``Gain and phase margin for multiloop {LQG} regulators,'' \emph{IEEE Transactions on Automatic Control}, vol.~22, no.~2, pp. 173--179, 1977.

\bibitem{lehtomaki1981robustness}
N.~Lehtomaki, N.~Sandell, and M.~Athans, ``Robustness results in linear-quadratic {Gaussian} based multivariable control designs,'' \emph{IEEE Transactions on Automatic Control}, vol.~26, no.~1, pp. 75--93, 1981.

\bibitem{hespanha2018linear}
J.~P. Hespanha, \emph{Linear Systems Theory}, 2nd~ed.\hskip 1em plus 0.5em minus 0.4em\relax Princeton University Press, 2018.

\bibitem{stevens_aircraft_2015}
B.~L. Stevens, F.~L. Lewis, and E.~N. Johnson, \emph{\BIBforeignlanguage{en}{Aircraft {Control} and {Simulation}: {Dynamics}, {Controls}, {Design}, and {Autonomous} {Systems}}}, 3rd~ed.\hskip 1em plus 0.5em minus 0.4em\relax Wiley-Blackwell, 2015.

\bibitem{andersonComputationalFluidDynamics2006}
J.~D. Anderson, \emph{Computational Fluid Dynamics: The Basics with Applications}, ser. {{McGraw-Hill}} Series in Aeronautical and Aerospace Engineering.\hskip 1em plus 0.5em minus 0.4em\relax McGraw-Hill, 2006.

\bibitem{andersonFundamentalsAerodynamics2017}
------, \emph{Fundamentals of Aerodynamics}, 6th~ed., ser. {{McGraw-Hill}} Series in Aeronautical and Aerospace Engineering.\hskip 1em plus 0.5em minus 0.4em\relax McGraw-Hill Education, 2017.

\bibitem{zipfelModelingSimulationAerospace2007}
P.~Zipfel and W.~Schiehlen, \emph{Modeling and Simulation of Aerospace Vehicle Dynamics}, 2nd~ed., ser. {{AIAA}} Education Series.\hskip 1em plus 0.5em minus 0.4em\relax {American Institute of Aeronautics and Astronautics}, 2007, vol.~54.

\bibitem{vrabie_optimal_2013}
D.~L. Vrabie, \emph{\BIBforeignlanguage{en}{Optimal {Adaptive} {Control} and {Differential} {Games} by {Reinforcement} {Learning} {Principles}}}, ser. {IET} control engineering series.\hskip 1em plus 0.5em minus 0.4em\relax London: The Institution of Engineering and Technology, 2013, no.~81.

\bibitem{kaelbling1996reinforcement}
L.~P. Kaelbling, M.~L. Littman, and A.~W. Moore, ``Reinforcement learning: A survey,'' \emph{Journal of artificial intelligence research}, vol.~4, pp. 237--285, 1996.

\bibitem{sutton1998reinforcement}
R.~S. Sutton and A.~G. Barto, \emph{Reinforcement Learning: An Introduction}, 2nd~ed.\hskip 1em plus 0.5em minus 0.4em\relax MIT Press, 2018.

\bibitem{lewis_reinforcement_2011}
F.~L. Lewis and K.~G. Vamvoudakis, ``Reinforcement learning for partially observable dynamic processes: {Adaptive} dynamic programming using measured output data,'' \emph{IEEE Transactions on Systems, Man, and Cybernetics, Part B (Cybernetics)}, vol.~41, no.~1, pp. 14--25, Feb. 2011.

\bibitem{hall2011reinforcement}
J.~Hall, C.~E. Rasmussen, and J.~Maciejowski, ``Reinforcement learning with reference tracking control in continuous state spaces,'' in \emph{2011 50th IEEE Conference on Decision and Control and European Control Conference}.\hskip 1em plus 0.5em minus 0.4em\relax IEEE, 2011, pp. 6019--6024.

\bibitem{vamvoudakis_q-learning_2017}
K.~G. Vamvoudakis, ``\BIBforeignlanguage{en}{Q-learning for continuous-time linear systems: {A} model-free infinite horizon optimal control approach},'' \emph{\BIBforeignlanguage{en}{Systems \& Control Letters}}, vol. 100, pp. 14--20, Feb. 2017.

\bibitem{luo2014q}
B.~Luo, D.~Liu, and T.~Huang, ``Q-learning for optimal control of continuous-time systems,'' \emph{arXiv preprint arXiv:1410.2954}, 2014.

\bibitem{lee2012integral}
J.~Y. Lee, J.~B. Park, and Y.~H. Choi, ``Integral {Q}-learning and explorized policy iteration for adaptive optimal control of continuous-time linear systems,'' \emph{Automatica}, vol.~48, no.~11, pp. 2850--2859, 2012.

\bibitem{lee2014integral}
------, ``On integral generalized policy iteration for continuous-time linear quadratic regulations,'' \emph{Automatica}, vol.~50, no.~2, pp. 475--489, 2014.

\bibitem{farjadnasab2022model}
M.~Farjadnasab and M.~Babazadeh, ``Model-free {LQR} design by {Q}-function learning,'' \emph{Automatica}, vol. 137, p. 110060, 2022.

\bibitem{horn2012matrix}
R.~A. Horn and C.~R. Johnson, \emph{Matrix Analysis}, 2nd~ed.\hskip 1em plus 0.5em minus 0.4em\relax Cambridge University Press, 2013.

\bibitem{kailath1980linear}
T.~Kailath, \emph{Linear Systems}.\hskip 1em plus 0.5em minus 0.4em\relax Prentice Hall, 1980.

\bibitem{filippov1988differential}
A.~F. Filippov, \emph{Differential Equations with Discontinuous Righthand Sides}.\hskip 1em plus 0.5em minus 0.4em\relax Kluwer Academic Publishers, 1988.

\bibitem{coddington1956theory}
E.~A. Coddington and N.~Levinson, \emph{Theory of Ordinary Differential Equations}.\hskip 1em plus 0.5em minus 0.4em\relax McGraw-Hill Publishing Co. Ltd., 1955.

\bibitem{royden1988real}
H.~L. Royden, \emph{Real Analysis}, 3rd~ed.\hskip 1em plus 0.5em minus 0.4em\relax Macmillan New York, 1988.

\bibitem{kolavr2025chain}
J.~Kol{\'a}{\v{r}} and O.~Maleva, ``Chain rule for pointwise {Lipschitz} mappings,'' \emph{arXiv preprint arXiv:2504.14385}, 2025.

\bibitem{serrin1969general}
J.~Serrin and D.~E. Varberg, ``A general chain rule for derivatives and the change of variables formula for the {Lebesgue} integral,'' \emph{The American Mathematical Monthly}, vol.~76, no.~5, pp. 514--520, 1969.

\bibitem{mathews2004numerical}
J.~Mathews, \emph{Numerical Methods Using {Matlab}}, 3rd~ed.\hskip 1em plus 0.5em minus 0.4em\relax Prentice Hall, 1999.

\bibitem{cichella2018bernstein}
V.~Cichella, I.~Kaminer, C.~Walton, N.~Hovakimyan, and A.~Pascoal, ``Bernstein approximation of optimal control problems,'' \emph{arXiv preprint arXiv:1812.06132}, 2018.

\bibitem{elnagar1995pseudospectral}
G.~Elnagar, M.~A. Kazemi, and M.~Razzaghi, ``The pseudospectral {Legendre} method for discretizing optimal control problems,'' \emph{IEEE Transactions on Automatic Control}, vol.~40, no.~10, pp. 1793--1796, 1995.

\bibitem{bonner2026pilot}
H.~M. Bonner and M.~R. Kirchner, ``Human as an actuator dynamic model identification,'' in \emph{2026 IEEE Aerospace Conference}, 2026, pp. 1--5.

\bibitem{proctor2016dynamic}
J.~L. Proctor, S.~L. Brunton, and J.~N. Kutz, ``Dynamic mode decomposition with control,'' \emph{{SIAM} Journal on Applied Dynamical Systems}, vol.~15, no.~1, pp. 142--161, 2016.

\bibitem{brunton2022data}
S.~L. Brunton and J.~N. Kutz, \emph{Data-Driven Science and Engineering: {Machine} Learning, Dynamical Systems, and Control}.\hskip 1em plus 0.5em minus 0.4em\relax Cambridge University Press, 2022.

\bibitem{willems2005note}
J.~C. Willems, P.~Rapisarda, I.~Markovsky, and B.~L. De~Moor, ``A note on persistency of excitation,'' \emph{Systems \& Control Letters}, vol.~54, no.~4, pp. 325--329, 2005.

\bibitem{van2020willems}
H.~J. Van~Waarde, C.~De~Persis, M.~K. Camlibel, and P.~Tesi, ``Willems’ fundamental lemma for state-space systems and its extension to multiple datasets,'' \emph{IEEE Control Systems Letters}, vol.~4, no.~3, pp. 602--607, 2020.

\bibitem{watkins1989learning}
C.~J. C.~H. Watkins, ``Learning from delayed rewards,'' Ph.D. dissertation, University of Cambridge, 1989.

\bibitem{watkins1992q}
C.~J. Watkins and P.~Dayan, ``Q-learning,'' \emph{Machine Learning}, vol.~8, pp. 279--292, 1992.

\bibitem{baird_reinforcement_1994}
L.~Baird, ``Reinforcement learning in continous time: {Advantage} updating,'' \emph{Proceedings of 1994 IEEE International Conference on Neural Networks (ICNN)}, vol.~4, pp. 2448--2453, 1994.

\bibitem{franklinFeedbackControlDynamic2018}
G.~F. Franklin, J.~D. Powell, and A.~Emami-Naeini, \emph{Feedback Control of Dynamic Systems}, 8th~ed.\hskip 1em plus 0.5em minus 0.4em\relax Pearson, 2018.

\bibitem{bowerfind2025model}
S.~R. Bowerfind, M.~R. Kirchner, G.~A. Hewer, D.~R. Robinson, P.~Chen, A.~Farahmandi, and K.~Estabridis, ``A model-free data-driven algorithm for continuous-time control,'' in \emph{2025 IEEE Aerospace Conference}, 2025, pp. 1--10.

\bibitem{wachter2006implementation}
A.~W{\"a}chter and L.~T. Biegler, ``On the implementation of an interior-point filter line-search algorithm for large-scale nonlinear programming,'' \emph{Mathematical Programming}, vol. 106, pp. 25--57, 2006.

\bibitem{andersson2019casadi}
J.~A. Andersson, J.~Gillis, G.~Horn, J.~B. Rawlings, and M.~Diehl, ``{CasADi}: {A} software framework for nonlinear optimization and optimal control,'' \emph{Mathematical Programming Computation}, vol.~11, pp. 1--36, 2019.

\bibitem{schmidtModernFlightDynamics2012}
D.~K. Schmidt, \emph{Modern Flight Dynamics}, 1st~ed.\hskip 1em plus 0.5em minus 0.4em\relax McGraw-Hill, 2012.

\bibitem{al37}
\emph{Freewing AL37 Airliner User Manual}, Freewing, 2025.

\bibitem{lpms-curs3}
\emph{LPMS-CURS3 Datasheet}, LP-RESEARCH Inc., Tokyo, Japan, 2025, available at \url{https://www.lp-research.com/wp-content/uploads/2025/10/20251016_LPMS-CURS3_EN.pdf}.

\bibitem{here3}
\emph{Here3/3+ Manual}, CubePilot, 2022, available at \url{https://docs.cubepilot.org/user-guides/here-3/here-3-manual}.

\bibitem{willems_least_1971}
J.~Willems, ``\BIBforeignlanguage{en}{Least squares stationary optimal control and the algebraic {Riccati} equation},'' \emph{\BIBforeignlanguage{en}{IEEE Transactions on Automatic Control}}, vol.~16, no.~6, pp. 621--634, Dec. 1971.

\end{thebibliography}


\begin{IEEEbiography}[{\includegraphics[width=1in,height=1.25in,clip,keepaspectratio]{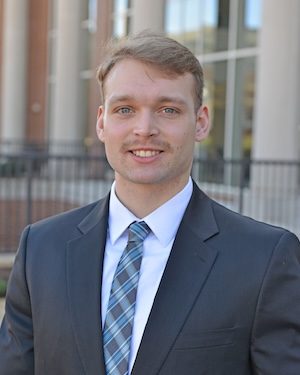}}]{Sean R. Bowerfind} is a Ph.D. Candidate in the Department of Electrical and Computer Engineering at Auburn University. His research areas include flight mechanics, optimal control, trajectory optimization, and applied mathematics. He received his B.S. in aerospace engineering from Auburn University in 2022 and his M.S. in aerospace engineering from Auburn University in 2023. He was the recipient of Robert G. Pitts Award for the Most Outstanding Senior in aerospace engineering and named the Undergraduate Student of the Year by the AIAA Greater Huntsville Section in 2022. Additionally, Sean is an officer in the United States Air Force, currently assigned to the Air Force Institute of Technology Civilian Institutions Program, and will attend Euro-NATO Joint-Jet Pilot Training upon completion of his Ph.D. 
\end{IEEEbiography}

\begin{IEEEbiography}[{\includegraphics[width=1in,height=1.25in,clip,keepaspectratio]{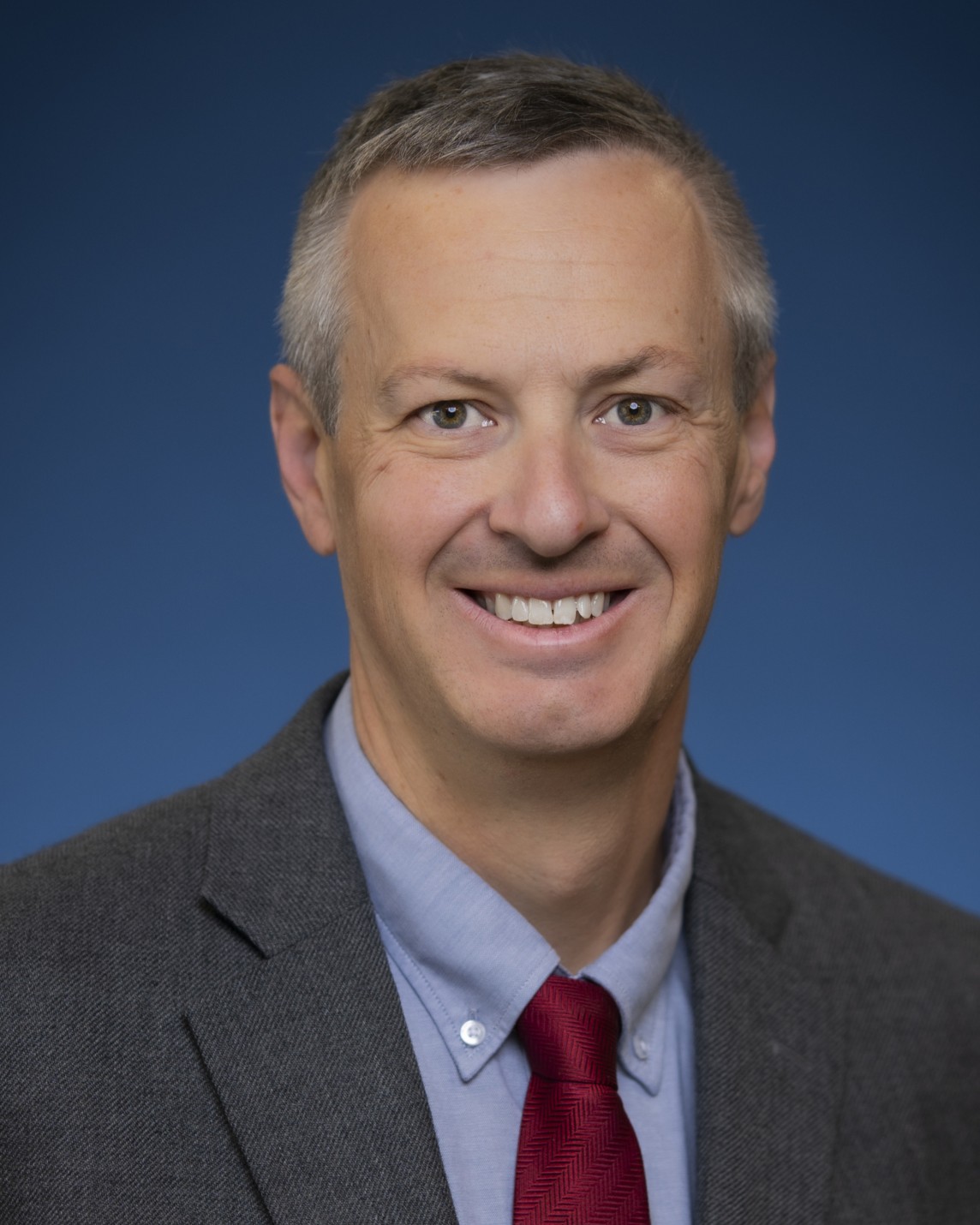}}]{Matthew R. Kirchner} (Member, IEEE) is the Godbold Endowed Assistant Professor of Electrical and Computer Engineering at Auburn University. He received the B.S. degree in mechanical engineering from Washington State University in 2007, the M.S. degree in electrical engineering from the University of Colorado at Boulder in 2013, and the Ph.D. degree in electrical and computer engineering from the University of California, Santa Barbara in 2023. He spent over 16 years at the Naval Air Warfare Center Weapons Division, China Lake, first joining the Navigation and Weapons Concepts Develop Branch in 2007 as a staff engineer. In 2012 he transferred into the Physics and Computational Sciences Division in the Research and Intelligence Department, where he served as a senior research scientist. His research interests include level set methods for optimal control, differential games, and reachability; multi-vehicle robotics; nonparametric signal and image processing; and navigation and flight control. He was the recipient of a Naval Air Warfare Center Weapons Division Graduate Academic Fellowship from 2010 to 2012; in 2011 was named a Paul Harris Fellow by Rotary International and in 2021 was awarded a Robertson Fellowship from the University of California in recognition of an outstanding academic record.
\end{IEEEbiography}

\begin{IEEEbiography}[{\includegraphics[width=1in,height=1.25in,clip,keepaspectratio]{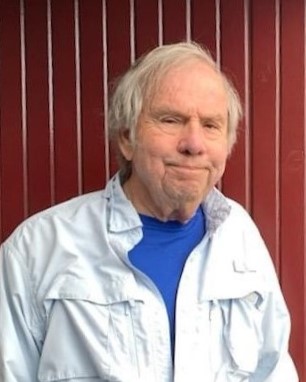}}]
{Gary A. Hewer} received his B.A. degree from Yankton College, South Dakota in 1962, and his M.S. degree and Ph.D. degree in mathematics from Washington State University in 1964 and 1968, respectively. Gary joined the Naval Air Warfare Center Weapons Division, China Lake, in 1968 and retired in 2025 as the NAVAIR Senior Scientist for Image and Signal Processing. In 1986 and 1987 he served as the interim Scientific Officer for Applied Analysis in the Mathematics Division at the Office of Naval Research (ONR). During his over 50 years of experience he has performed research on control, radar guidance, wavelet applications for both compression and small target detection, image registration and processing, probability theory, and autonomy. In 1987, he received the then Naval Weapons Center’s Technical Director’s Award for his work in control theory and was elected a Senior Fellow of the Naval Weapons Center in 1990 for his work in radar tracking, target modeling, and control theory. In 1998, he
was awarded the Navy Meritorious Civilian Award for his contributions as a Navy research scientist, and in 2002, he was inducted as a NAVAIR fellow. Recently, he was named an Esteemed NAVAIR Fellow.
\end{IEEEbiography}

\balance

\end{document}